\documentclass[12pt]{amsart}
\usepackage{amssymb}
\usepackage{t1enc}
\usepackage[mathscr]{eucal}
\usepackage{parskip}
\usepackage{xcolor}
\usepackage{tikz}
\usepackage{tikz-cd}
\usepackage{tikz-3dplot}
\usetikzlibrary{calc}
\usepackage{pgfplots}
\usepackage{hyperref}
\numberwithin{equation}{section}
\usepackage[margin=2.9cm]{geometry}
\usepackage{booktabs}
\usepackage{enumerate}
\usepackage{mathtools}
\usepackage{bm}
\usepackage{xfrac}
\usepackage{adjustbox}
\usepackage{subcaption}

\usepackage{comment}
\usepackage{nicematrix}

\usepackage{letltxmacro}
\LetLtxMacro\orgvdots\vdots
\LetLtxMacro\orgddots\ddots

\makeatletter
\DeclareRobustCommand\vdots{%
  \mathpalette\@vdots{}%
}
\newcommand*{\@vdots}[2]{%
  \sbox0{$#1\cdotp\cdotp\cdotp\m@th$}%
  \sbox2{$#1.\m@th$}%
  \vbox{%
    \dimen@=\wd0 %
    \advance\dimen@ -3\ht2 %
    \kern.5\dimen@
    \dimen@=\wd2 %
    \advance\dimen@ -\ht2 %
    \dimen2=\wd0 %
    \advance\dimen2 -\dimen@
    \vbox to \dimen2{%
      \offinterlineskip
      \copy2 \vfill\copy2 \vfill\copy2 %
    }%
  }%
}
\DeclareRobustCommand\ddots{%
  \mathinner{%
    \mathpalette\@ddots{}%
    \mkern\thinmuskip
  }%
}
\newcommand*{\@ddots}[2]{%
  \sbox0{$#1\cdotp\cdotp\cdotp\m@th$}%
  \sbox2{$#1.\m@th$}%
  \vbox{%
    \dimen@=\wd0 %
    \advance\dimen@ -3\ht2 %
    \kern.5\dimen@
    \dimen@=\wd2 %
    \advance\dimen@ -\ht2 %
    \dimen2=\wd0 %
    \advance\dimen2 -\dimen@
    \vbox to \dimen2{%
      \offinterlineskip
      \hbox{$#1\mathpunct{.}\m@th$}%
      \vfill
      \hbox{$#1\mathpunct{\kern\wd2}\mathpunct{.}\m@th$}%
      \vfill
      \hbox{$#1\mathpunct{\kern\wd2}\mathpunct{\kern\wd2}\mathpunct{.}\m@th$}%
    }%
  }%
}
\makeatother

\usepackage{pgfplots}
\usepgfplotslibrary{polar}

\usepackage{listings}
\theoremstyle{plain}
\newtheorem{thm}{Theorem}[section]
\newtheorem{lemma}[thm]{Lemma}
\newtheorem{cor}[thm]{Corollary}
\newtheorem{prop}[thm]{Proposition}

 \theoremstyle{definition}
\newtheorem{defn}[thm]{Definition}
\newtheorem{rem}[thm]{Remark}
\newtheorem{ex}[thm]{Example}

\newtheorem{setup}[thm]{Setup}

\newcommand{\op}[1]{\operatorname{#1}}
\newcommand{\mb}[1]{\mathbb{#1}}
\newcommand{\mc}[1]{\mathcal{#1}}

\newcommand{\mr}[1]{\mathrm{#1}}

\newcommand{\vphi}{\varphi}

\newcommand{\GW}{\mathrm{GW}}

\newcommand{\Tr}{\operatorname{Tr}}
\newcommand{\N}{\operatorname{N}}
\newcommand{\rank}{\operatorname{rank}}

\newcommand{\Disc}{\operatorname{Disc}}
\newcommand{\ind}{\operatorname{ind}}
\newcommand{\G}{\mathrm{G}}
\newcommand{\gauss}{\mathcal G}
\newcommand{\Res}{\operatorname{Res}}
\renewcommand{\P}{\mathbb{P}}
\newcommand{\Mor}{\operatorname{Mor}}
\newcommand{\Sym}{\operatorname{Sym}}
\newcommand{\PGL}{\operatorname{PGL}}
\newcommand{\Aut}{\operatorname{Aut}}
\newcommand{\seg}{\operatorname{seg}}
\newcommand{\con}{\operatorname{con}}
\newcommand{\FK}{\operatorname{FK}}
\newcommand{\Conf}{\operatorname{Conf}}
\newcommand{\Gal}{\operatorname{Gal}}
\renewcommand{\emptyset}{\varnothing}

\newcommand{\R}{\mathbb{R}}
\newcommand{\C}{\mathbb{C}}
\newcommand{\VB}{\mathtt{V}_B}
\newcommand{\LB}{\mathtt{L}_B}
\newcommand{\RB}{\mathtt{R}_B}
\newcommand{\KB}{\mathtt{K}_B}

\newcommand{\detABQ}{\mc{A}(B,Q)}
\newcommand{\A}{\mathbb{A}}
\definecolor{darkspringgreen}{rgb}{0.09, 0.6, 0.1}
\definecolor{dsg}{rgb}{0.09, 0.6, 0.1}

\makeatletter
\@namedef{subjclassname@2020}{\textup{2020} Mathematics Subject Classification}
\makeatother

\begin{document}
\title{Quadratic Segre indices}

\author[Espreafico]{Felipe Espreafico}
\address{Institute de Mathématiques de Jussieu \\ Sorbonne Université}
\email{espreafico-guelerman@imj-prg.fr}
\urladdr{espreafico-guelerman.github.io}

\author[McKean]{Stephen McKean}
\address{Department of Mathematics \\ Brigham Young University} 
\email{mckean@math.byu.edu}
\urladdr{shmckean.github.io}

\author[Pauli]{Sabrina Pauli}
\address{Fachbereich Mathematik \\ TU Darmstadt}
\email{pauli@mathematik.tu-darmstadt.de}
\urladdr{homepage.sabrinapauli.com}


\begin{abstract}
We prove that the local Euler class of a line on a degree $2n-1$ hypersurface in projective $n+1$ space is given by a product of indices of Segre involutions. Segre involutions and their associated indices were first defined by Finashin and Kharlamov over the reals. Our result is valid over any perfect field of characteristic not 2 and gives an infinite family of problems in enriched enumerative geometry with a shared geometric interpretation for the local type.
\end{abstract}

\maketitle

\section{Introduction}
We begin with a brief statement of our main result. We then give the motivation for our result, as well as the relevant terminology, in Section~\ref{sec:motivation}. Sufficiently motivated readers who are tired of history may skip to Section~\ref{sec:ideas}, where we summarize the ideas behind our proofs.

In short, we give a geometric interpretation for the enumerative weight of a line on a general degree $2n-1$ hypersurface in $\mb{P}^{n+1}$ over a perfect field of characteristic not 2, generalizing work of Kass--Wickelgren in the $n=2$ case \cite{cubicsurface} and Pauli in the $n=3$ case \cite{PauliQuadraticTypes}. Our geometric interpretation is directly inspired by a construction of Finashin and Kharlamov \cite{FinashinKharlamov2021}.

\begin{setup}
    Let $X\subset\mb{P}^{n+1}_k$ be a general hypersurface of degree $2n-1$ over a perfect field $k$. Given a line $\ell\subset X$ with field of definition $k(\ell)$, there is an associated rational curve $\mc{G}(\ell)\subset\mb{P}^{n-1}_k$ of degree $2n-2$. Over $\overline{k}$ there are exactly $\binom{n}{2}$ planes of dimension $n-3$ meeting $\mc{G}(\ell)$ in $2n-4$ points. Each such $(2n-4)$-secant plane $S$, whose field of definition we denote $k(S)$, determines an involution $i_S:\ell_{k(S)}\to\ell_{k(S)}$. The fixed points of $i_S$ are defined over $k(S)(\sqrt{\alpha_S})$ for some $\alpha_S\in k(S)^\times/(k(S)^\times)^2$.

    Let $\GW(k)$ denote the Grothendieck--Witt ring of virtual quadratic forms over $k$. Given $a\in k^\times$, let $\langle a\rangle\in\GW(k)$ represent the isomorphism class of the bilinear form $[(x,y)\mapsto axy]$. Given a field extension $L/k$, let $\N_{L/k}:L\to k$ and $\Tr_{L/k}:L\to k$ denote the field norm and field trace, respectively. 
    
    We define the \emph{Segre index} of $\ell\subset X$ to be
    \[\seg(X,\ell):=\Tr_{k(\ell)/k}\big\langle\prod_S \N_{k(S)/k(\ell)}(\alpha_S)\big\rangle\in\GW(k),\]
    where this product ranges over the $(n-3)$-planes that are $(2n-4)$-secant to $\mc{G}(\ell)$. Note that $k(\ell)/k$ is separable since $k$ is perfect, so the field trace here is non-degenerate.
\end{setup}

\begin{thm}[Main Theorem]\label{thm:local index}
    Let $X=\mb{V}(F)\subset\mb{P}^{n+1}_k$ be a general hypersurface of degree $2n-1$ over a perfect field $k$ of characteristic not 2. Let
    \begin{align*}
    \sigma_F:\mb{G}(1,n+1)\to\Sym^{2n-1}(\mc{S}^\vee)
    \end{align*}
    denote the section determined by $F$, where $\mc{S}$ is the tautological bundle on the Grassmannian of lines in $\mb{P}^{n+1}$. Let $\ind_\ell\sigma_F\in\GW(k)$ denote the local index of $\sigma_F$ at a line $\ell\subset X$. Then
    \[\ind_\ell\sigma_F=\seg(X,\ell).\]
\end{thm}

Theorem~\ref{thm:local index} allows us to immediately deduce a quadratically enriched count of lines on a degree $2n-1$ hypersurface in $\mb{P}^{n+1}$. This is because the Euler number underlying the total count is completely determined by the (signed) real and complex counts \cite[Example 8.2]{LevineMotivicEulerCharacteristic}, \cite[Corollary~6.9]{BW23}, the real count is $(2n-1)!!$ \cite{FinashinKharlamov2012,OkonekTeleman}, and the complex count is computable by Schubert calculus.

\begin{thm}\label{thm:arithmetic count}
    Let $X\subset\mb{P}^{n+1}_k$ be a general hypersurface of degree $2n-1$ over a field $k$ of characteristic 0.\footnote{The only place where we use the assumption $\op{char}{k}=0$ is Proposition~\ref{prop: castelnuovo}.} Let $c(n)$ denote the top Chern number of $\Sym^{2n-1}(\mc{S}^\vee)\to\mb{G}(1,n+1)$. Then we have an equality
    \[\sum_{\ell\subset X}\seg(X,\ell)=(2n-1)!!\langle 1\rangle+\frac{c(n)-(2n-1)!!}{2}\mb{H}\]
    in $\GW(k)$. 
\end{thm}

These theorems provide an infinite family of enriched enumerative problems with a shared geometric interpretation for their local indices. In the language of the \emph{geometricity question} \cite[Appendix C]{McK22}, these ``lines on hypersurfaces'' problems belong to the same phylum of enumerative problems, as one would expect.

\subsection{Motivation}\label{sec:motivation}
Famously, there are 27 lines on every smooth complex cubic surface \cite{cayley}. A modern method for calculating the number 27 is as the top Chern number of the bundle
\[\Sym^3(\mc{S}^\vee)\to\mb{G}(1,3),\]
where $\mc{S}\to\mb{G}(1,3)$ is the tautological bundle on the Grassmannian of projective lines in $\mb{P}^3$. More generally, one can prove that there are finitely many lines on a generic degree $2n-1$ hypersurface in $\mb{P}^{n+1}$, all of which are reduced. The top Chern number of
\[\Sym^{2n-1}(\mc{S}^\vee)\to\mb{G}(1,n+1)\]
can then be computed via Schubert calculus, giving the 2875 lines on a quintic threefold or, for those impressed by large numbers, the 1,192,221,463,356,102,320,754,899 lines on a novemdecic tenfold (see e.g.~\cite[\S 6.5]{3264}).

Lines on hypersurfaces generally make for good conversation with non-mathematicians, such as relatives or university administrators --- most people can appreciate the visual beauty of a cubic surface with its configuration of lines, and administrators tend to be impressed by large numbers. However, as with all online image searches, one should be careful when pulling up a picture of a cubic surface and its lines at a family party. This is because any such picture will actually depict a \emph{real} cubic surface, which may have fewer than 27 \emph{real} lines.

\subsubsection{The real story}\label{sec:real story}
Schl\"afli proved that every smooth real cubic surface has 3, 7, 15, or 27 real lines \cite{Sch58}. Nearly a century later, Segre constructed an involution on each real line on a cubic surface \cite{Segre1942}, which we now describe. Let $\ell$ be a real line on a real cubic surface $X$. For each point $p\in\ell$, the intersection $T_pX\cap X$ is the union of $\ell$ and a residual conic $Q_p\subset T_pX$. The intersection $Q_p\cap\ell$ consists of $p$ and another point, say $q$. The \emph{Segre involution} swaps $p$ and $q$ (see Figure~\ref{fig:cubic}).

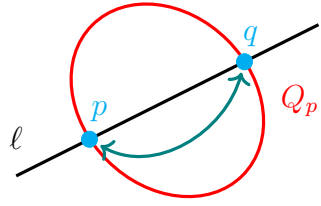
\begin{figure}
\centering
\begin{tikzpicture}
    \draw[very thick] (-2,-1) -- (2,1);
    \node at (-2,-.5) {$\ell$};
    \draw[red,very thick,rotate=45] (0,0) ellipse (32pt and 40pt);
    \node[red] at (1.75,0) {$Q_p$};
    \node (p) at (-1.03,-.515) {};
  \node (q) at (1.02,.51) {};
  \draw[cyan,fill=cyan] (p) circle (3pt);
  \draw[cyan,fill=cyan] (q) circle (3pt);
  \draw [<->,very thick,teal] (p) to [out=-30,in=-105] (q);
  \node[cyan] at (-.9,-.15) {$p$};
  \node[cyan] at (1.1,.85) {$q$};
\end{tikzpicture}
\caption{Segre involution on a cubic surface}\label{fig:cubic}
\end{figure}

Define \emph{hyperbolic lines} as those whose Segre involution has real fixed points and \emph{elliptic lines} as those whose Segre involution has complex fixed points. For each topological type of real cubic surface, Segre computed the number of hyperbolic and elliptic lines contained therein. These computations imply the striking formula
\begin{equation}\label{eq:real cubics}
\#\{\text{hyperbolic lines}\}-\#\{\text{elliptic lines}\}=3,
\end{equation}
although this seems to have gone unnoticed until it was observed by Finashin--Kharlamov \cite{FinashinKharlamov2012} and Okonek--Teleman \cite{OkonekTeleman} about 70 years later. Equation~\ref{eq:real cubics} should be viewed as the correct analog of the 27 lines on a cubic surface over $\mb{C}$. While the total number of real lines on a cubic surface over $\mb{R}$ depends on the choice of cubic surface, the \emph{signed} count of lines given in Equation~\ref{eq:real cubics} is independent of the choice of cubic surface.

Inspired by Equation~\ref{eq:real cubics}, Finashin and Kharlamov set out to give a signed count of real lines on degree $2n-1$ hypersurfaces in $\mb{RP}^{n+1}$. They proved that the overall signed count is equal to $(2n-1)!!$ \cite{FinashinKharlamov2012}, which arises as the Euler number of the real vector bundle $\Sym^{2n-1}(\mc{S}^\vee)\to\mb{G}(1,n+1)$. See also \cite{OkonekTeleman,Sol06} for related results. To complete the signed count of real lines on hypersurfaces, there also needs to be a geometric determination of the type of a line, analogous to Segre involutions determining whether a line on a cubic surface is hyperbolic or elliptic. Finashin and Kharlamov gave two such geometric interpretations (and proved that they are equivalent) in \cite{FinashinKharlamov2021}: \emph{Welschinger weights} and \emph{Segre indices}. We will only discuss Segre indices, as these are the interpretation relevant for our article.

Given a smooth hypersurface $X\subset\mb{R}\mb{P}^{n+1}$, there is a \emph{Gauß map}
\begin{align*}
    \mc{G}:X&\to\mb{RP}^{n+1}\\
    p&\mapsto T_pX,
\end{align*}
where the target $\mb{R}\mb{P}^{n+1}$ is the dual projective space parameterizing codimension 1 hyperplanes in $\mb{R}\mb{P}^{n+1}$. If we restrict $\mc{G}$ to a line $\ell\subset X$, then the tangent plane $\mc{G}(p):=T_pX$ contains $\ell$ for each $p\in\ell$. The flag variety of codimension 1 hyperplanes in $\mb{R}\mb{P}^{n+1}$ containing a given line (in this case, $\ell$) is isomorphic to $\mb{R}\mb{P}^{n-1}$. We thus obtain a restricted Gauß map
\[\mc{G}|_\ell:\ell\to\mb{R}\mb{P}^{n-1}.\]
The image $\mc{G}(\ell)\subset\mb{R}\mb{P}^{n-1}$ is a rational curve of degree $2n-2$, which we refer to as the \emph{Gauß curve} of $\ell$. For $n\geq 3$, a generic degree $2n-2$ rational curve in $\mb{R}\mb{P}^{n-1}$ has a finite number of $(n-3)$-dimensional $(2n-4)$-secants; in fact, there are $\binom{n}{2}$ such secants defined over $\mathbb{C}$. 

Each secant defined over $\mathbb{R}$ determines a \emph{Segre involution} as follows. Let $S$ be an $(n-3)$-plane meeting $\mc{G}(\ell)$ in $2n-4$ points defined over $\mathbb{R}$. There is a pencil of $(n-2)$-dimensional (i.e.~codimension 1 in $\mb{P}^{n-1}$) hyperplanes containing $S$. Each $(n-2)$-plane in this pencil meets the Gauß curve $\mc{G}(\ell)$ in $2n-2$ points, by B\'ezout's theorem: the $2n-4$ secant points given by $S\cap\mc{G}(\ell)$, and another pair of points. The Segre involution associated to $S$ swaps this residual pair of points, as depicted in Figure~\ref{fig:segre}.\footnote{The plane quartic in Figure~\ref{fig:n=3} is the well-known ampersand curve. To construct a rational space sextic curve (like the one in Figure~\ref{fig:n=4}) and its six quadrisecants, it helps to use a remarkable theorem of Dye stating that such quadrisecants form half of a double six of lines on a cubic surface containing the sextic \cite{Dye97}. See Remark~\ref{rem:quadrisecants} for more details.\label{fn:curves}}

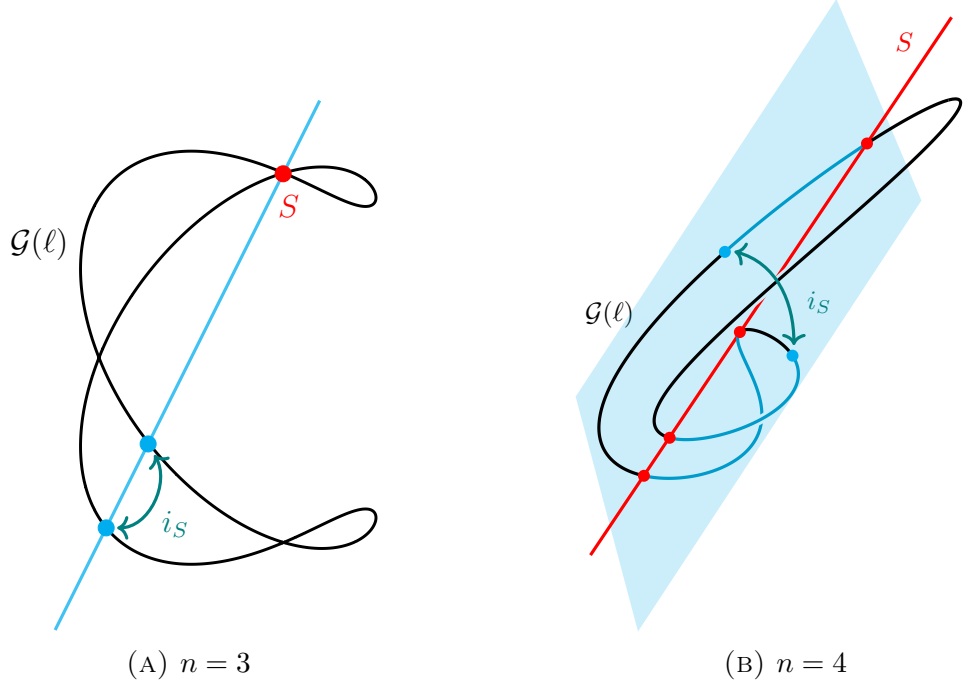
\begin{figure}
\begin{subfigure}{.5\linewidth}
\centering
\begin{tikzpicture}
  \node[inner sep=0pt] (amp) at (0,0){\includegraphics[width=112.317335pt]{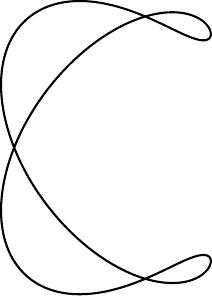}};
  \node[red] at (.8,2) {$S$};
  \node at (-2.5,1.5) {$\mc{G}(\ell)$};
  \node (p) at (.71-1.77,2.4-2*1.77) {};
  \node (q) at (.71-2.325,2.4-2*2.325) {};
  \draw[very thick,cyan,opacity=.75] (.71+.5,2.4+1) -- (.71-3,2.4-6);
  \draw[red,fill=red] (.726,2.432) circle (3pt);
  \draw[cyan,fill=cyan] (p) circle (3pt);
  \draw[cyan,fill=cyan] (q) circle (3pt);
  \draw [<->,very thick,teal] (p) to [out=300,in=0] (q);
  \node[teal] at (-.7,-2.2) {$i_S$};
\end{tikzpicture}
\caption{$n=3$}\label{fig:n=3}
\end{subfigure}%
\begin{subfigure}{.5\linewidth}
\centering
\includegraphics{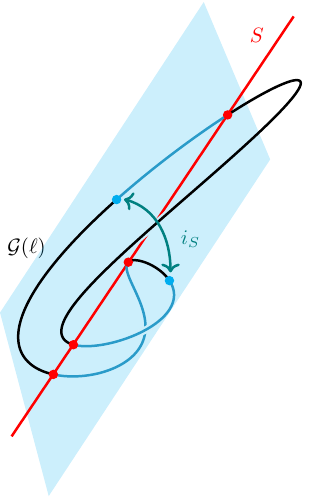}
\caption{$n=4$}\label{fig:n=4}
\end{subfigure}
\caption{Segre involutions}\label{fig:segre}
\end{figure}

The \emph{Segre index} of $\ell$, which we denote $\seg(X,\ell)$, is a product of indices associated to the Segre involutions on $\mc{G}(\ell)$ corresponding to the secant defined over $\mathbb{R}$. This Gauß curve is the embedding of a conic $Q_\ell\subset\mb{R}\mb{P}^2$, and the Segre involution associated to a secant $S$ pulls back to an involution on $Q_\ell$. Define the index of this involution, denoted $\seg_S(X,\ell)$, to be $+1$ if its fixed points are real and $-1$ if its fixed points are complex. Then
\[\seg(X,\ell):=\prod_{\text{real secants }S}\seg_S(X,\ell),\]
and Finashin and Kharlamov proved that
\begin{equation}\label{eq:fk signed count}
\sum_{\ell\subset X}\seg(X,\ell)=(2n-1)!!
\end{equation}
for a generic degree $2n-1$ hypersurface in $\mb{RP}^{n+1}$. Note that Equation~\ref{eq:real cubics} is the $n=2$ case of Equation~\ref{eq:fk signed count}. For $n=2$, the Segre index agrees with the convention that assigns $+1$ to hyperbolic lines and $-1$ to elliptic lines on a cubic surface.

\subsubsection{The arithmetic story}
Our goal (stated previously as Theorem~\ref{thm:arithmetic count}) is to prove an analog of Equation~\ref{eq:fk signed count} over an arbitrary base field. The framework for such a theorem is the \emph{enriched enumerative geometry} program, in which Kass and Wickelgren's count of lines on cubic surfaces is an early seminal result \cite{cubicsurface}.

Let $k$ be a field. We assume that $\operatorname{char}{k}\neq 2$, as we want involutions to behave well over $k$. Let $\GW(k)$ denote the Grothendieck--Witt ring of isomorphism classes of symmetric non-degenerate bilinear forms. This ring is generated by classes of the form $\langle a\rangle = [(x,y)\mapsto axy]$ for $a\in k^\times$. Kass and Wickelgren defined an Euler ``number'' $e(\Sym^3(\mc{S}^\vee))\in\GW(k)$ that satisfies a Poincar\'e--Hopf theorem: given a smooth cubic surface $X$ over $k$, the associated section $\sigma_X:\mb{G}(1,3)\to\Sym^3(\mc{S}^\vee)$ has local indices $\ind_\ell\sigma_X\in\GW(k)$ such that
\begin{equation}\label{eq:partial arith cubic}
\sum_{\ell\subset X}\ind_\ell\sigma_X=e(\Sym^3(\mc{S}^\vee)).
\end{equation}
In order to turn Equation~\ref{eq:partial arith cubic} into an enumerative theorem, one needs to compute the Euler number on the right hand side and give a geometric interpretation of the local indices on the left hand side. Kass and Wickelgren computed
\[e(\Sym^3(\mc{S}^\vee))=15\langle 1\rangle+12\langle -1\rangle\]
and proved that $\ind_\ell\sigma_X$ is determined by the Segre involution associated to $\ell$. Just as the Segre involution of a real line has either real or complex fixed points, the fixed points of the Segre involution on a line with field of definition $k(\ell)$ are defined over $k(\ell)(\sqrt{\alpha_\ell})$ for some $\alpha_\ell\in k(\ell)^\times/(k(\ell)^\times)^2$. Kass and Wickelgren proved that
\[\ind_\ell\sigma_X=\Tr_{k(\ell)/k}\langle \alpha_\ell\rangle\in\GW(k),\]
giving the desired geometric description of the relevant local indices. Altogether, we get the count
\begin{equation}\label{eq:arith cubic}
\sum_{\ell\subset X}\Tr_{k(\ell)/k}\langle \alpha_\ell\rangle=15\langle 1\rangle+12\langle -1\rangle.
\end{equation}
One of the key features of Equation~\ref{eq:arith cubic} is that it generalizes both the complex and real counts of lines on cubic surfaces. The rank of Equation~\ref{eq:arith cubic} states that
\[\sum_{\ell\subset X_{\overline{k}}}1=27,\]
counting 27 geometric lines on every smooth cubic surface. When $k$ admits a real embedding, the signature of Equation~\ref{eq:arith cubic} exactly recovers Equation~\ref{eq:real cubics}, as hyperbolic and elliptic lines respectively satisfy $\alpha_\ell=+1$ and $\alpha_\ell=-1$, while the signature of $\Tr_{\mb{C}/\mb{R}}\langle\alpha_\ell\rangle$ is 0 for any complex line.

Lines on quintic threefolds over general fields were considered by Levine and Pauli. Levine \cite[Example 8.3]{LevineMotivicEulerCharacteristic} computed
\[e(\Sym^5(\mc{S}^\vee))=1445\langle 1\rangle+1430\langle -1\rangle,\] 
while Pauli \cite{PauliQuadraticTypes} provided a geometric interpretation of the local indices contributing to this Euler number. The geometric interpretation is a quadratic enrichment of (the $n=3$ case of) the Segre index described in Section~\ref{sec:real story}. In this case, the Gauß curve $\mc{G}(\ell)$ associated to a line $\ell$ is a plane quartic, and geometrically there are $\binom{3}{2}=3$ planes of dimension $3-3=0$ that meet this quartic with order $2\cdot 3-4=2$. In other words, $\mc{G}(\ell)$ is a plane quartic with 3 nodes. The Segre involution 
\[i_\nu:\mc{G}(\ell)\to\mc{G}(\ell)\]
associated to a node $\nu\in\mc{G}(\ell)$ is given by taking a pencil of lines through $\nu$, intersecting each line $L$ with $\mc{G}(\ell)$, and swapping the pair of points in $\mc{G}(\ell)\cap L-\nu$ (see Figure~\ref{fig:n=3}, where $S=\nu$). The fixed points of $i_\nu$ are defined over $k(\nu)(\sqrt{\alpha_\nu})$ for some $\alpha_\nu\in k(\nu)^\times/(k(\nu)^\times)^2$, and the Segre index of $\ell$ is defined as
\[\seg(X,\ell):=\Tr_{k(\ell)/k}\big\langle\prod_{\nu} \N_{k(\nu)/k(\ell)}(\alpha_{\nu})\big\rangle.\]

\begin{rem}
While we have described the Segre involution associated to a node $\nu$ as an involution of the form $i_\nu:\mc{G}(\ell)\to\mc{G}(\ell)$, one can equivalently consider the pullback of $i_\nu$ to an involution $\ell$. We will prove that the fixed points of the involutions on $\mc{G}(\ell)$ and $\ell$ have isomorphic residue fields. This will imply that $\seg(X,\ell)$ is independent of whether we think of the involution occurring on $\mc{G}(\ell)$ or $\ell$, so we will generally not distinguish between these two perspectives.
\end{rem}

Already at the time of \cite{PauliQuadraticTypes}, it was expected that a similar description of the local index in terms of Segre involutions should hold in the general case of lines on degree $2n-1$ hypersurfaces in $\mb{P}^{n+1}$. However, the proof in \cite{PauliQuadraticTypes} utilizes the fact that for a general rational plane quartic, its trio of nodes lie in general position (i.e.~not on a line). It is not true that the $(n-3)$-dimensional $(2n-4)$-secants to a general degree $2n-2$ space curve in $\mb{P}^{n-1}$ are in general position when $n>3$, since $\binom{n}{2}>\dim\mb{G}(n-3,n-1)+1$ in this range. Furthermore, the proofs in \cite{FinashinKharlamov2021} rely on wall-crossing arguments, which do not readily generalize to other fields.

As we have claimed in Theorem~\ref{thm:local index}, the quadratic Segre index of a line indeed computes the local index contributing to the Euler class $e(\Sym^{2n-1}(\mc{S}^\vee))$. In the next subsection, we will outline the proof ideas that allowed us to sidestep the difficulties that arise when trying to generalize from $\mb{R}$ to $k$ or from $n=3$ to $n>3$.

\subsection{Ideas behind the proof}\label{sec:ideas}
As previously mentioned, we use a construction of Finashin and Kharlamov \cite[Section 6]{FinashinKharlamov2021} to relate lines on hypersurfaces to involutions arising from the secants of their associated Gauß curves, and equivalently to involutions associated to plane conics avoiding an auxiliary set of points. This construction allows us to generalize beyond the cases of cubic surfaces \cite{cubicsurface} and quintic threefolds \cite{PauliQuadraticTypes}.

In order to show that the local index of a line is equal to the product of Segre indices associated to the Gauß curve or plane conic, we essentially need to show that two determinants are equal up to squares\footnote{We say that two elements $x,y$ of a field $K$ are \emph{equal up to squares} if there exists $c\in K^\times$ such that $x=c^2y$.}. Over $\mb{R}$, this boils down to showing that two real numbers have the same sign; this is where Finashin--Kharlamov utilize a wall-crossing argument. Over arbitrary fields, we instead will show that these two determinants are given by regular maps with the same zero locus. This step more or less falls out of the geometry, which will allow us to express these regular maps in terms of resultants. The primary difficulties arise from the fact that conics over non-closed fields need not have rational points. We overcome this difficulty by modifying Finashin and Kharlamov's construction and working with \emph{parameterized} conics.

The final step is to show that the relevant regular maps agree at a point outside of their zero locus. This is done with an elementary, although somewhat lengthy, argument involving only basic linear algebra.

\subsection{Outline}
Here is a quick outline of the article.
\begin{itemize}
    \item In Section~\ref{sec:local index}, we discuss the bundle $\Sym^{2n-1}(\mc{S}^\vee)\to\mb{G}(1,n+1)$, the section induced by a hypersurface, our choice of coordinates on $\mb{G}(1,n+1)$, and our local trivializations of $\Sym^{2n-1}(\mc{S}^\vee)$. We then compute the local index and give examples discussing the cases of cubic surfaces and quintic threefolds.
    \item In Section~\ref{sec: types of lines 2}, we discuss Finashin and Kharlamov's conic models for rational space curves. This includes a dictionary translating between Gauß curves with their secants and plane conics with their associated loci of points. We also prove that the index of a Gauß curve can be computed in terms of a conic model.
    \item We conclude in Section~\ref{sec:local = segre} by proving that the local index (given in Section~\ref{sec:local index}) is equal to the index of a conic model (given in Section~\ref{sec: types of lines 2}), thereby proving Theorem~\ref{thm:local index}.
\end{itemize}

\subsection{Conventions}
We assume throughout that $k$ is a perfect field of characteristic not 2. This is the generality in which we prove Theorem~\ref{thm:local index}. Theorem~\ref{thm:arithmetic count} is a corollary of Theorem~\ref{thm:local index}, but we need to assume $\op{char}{k}=0$ for this corollary in order to apply Kleiman's transversality theorem \cite{Kle74}. Conjecturally, one should be able to remove this assumption (see Remark~\ref{rem:sottile}).

Whenever we write $\overline{k}$, we mean some fixed algebraic closure of $k$.

\subsection*{Acknowledgements}
We thank Sergey Finashin and Viatcheslav Kharlamov for helpful correspondence, including explaining the ideas behind Lemma~\ref{lem:generic} to us. We also thank Kirsten Wickelgren for helpful discussions.

We thank the anonymous referee for pointing out a mistake and for their extensive helpful feedback.

Felipe Espreafico acknowledges support of the European Research Council through the grant ROGW-864919.

Stephen McKean received support from the National Science Foundation (DMS-2502365) and the Simons Foundation.

Sabrina Pauli acknowledges support by Deutsche Forschungsgemeinschaft (DFG, German Research Foundation) through the Collaborative
Research Centre TRR 326 Geometry and Arithmetic of Uniformized Structures, project number 444845124.

\section{The local index}\label{sec:local index}
In this section, we will recall the definition and computation of the local index of a line on a degree $2n-1$ hypersurface in $\mb{P}^{n+1}$. The procedure for computing this index is standard, so the material in this section will be standard as well. It is the interpretation of this index in terms of the geometry at hand that is interesting and will comprise the balance of the article.

Let $k$ be a field. Let $\mb{G}(1,n+1)$ denote the Grassmannian of projective lines in $\mb{P}^{n+1}_k$. This Grassmannian is isomorphic to $\G(2,n+2)$, the Grassmannian of affine 2-planes in affine $(n+2)$-space. Let $\mc{S}\to\G(2,n+2)$ denote the tautological bundle, which we can also view as a rank 2 vector bundle $\mc{S}\to\mb{G}(1,n+1)$. The rank of the bundle
\[\Sym^{2n-1}(\mc{S}^\vee)\to\mb{G}(1,n+1)\]
is equal to $\binom{2n-1+2-1}{2n-1}=2n=\dim\mb{G}(1,n+1)$, so a generic section of this bundle should have a finite vanishing locus. Moreover, any degree $2n-1$ form $F$ on $\mb{P}^{n+1}_k$ determines a section
\begin{align*}
    \sigma_F:\mb{G}(1,n+1)&\to\Sym^{2n-1}(\mc{S}^\vee)\\
    \ell&\mapsto F|_\ell
\end{align*}
that vanishes precisely on those lines contained in the hypersurface $\mb{V}(F)$. One can then obtain an enriched count of the lines on $\mb{V}(F)$ by computing the Euler number $e(\Sym^{2n-1}(\mc{S}^\vee))\in\GW(k)$, using a Poincar\'e--Hopf formula to express $e(\Sym^{2n-1}(\mc{S}^\vee))$ as a sum of local degrees of $\sigma_F$ over its vanishing locus, and giving a geometric interpretation to these local indices. The Euler number and Poincar\'e--Hopf formula for this bundle are given in \cite[Theorem~1.1 and Corollary~6.9]{BW23}. 

In order to compute the local indices of $\sigma_F$, we need to make a suitable choice of coordinates on $\mb{G}(1,n+1)$ and local trivializations of $\Sym^{2n-1}(\mc{S}^\vee)$. For coordinates, we give coordinates centered about each line. Given a line $\ell\in\mb{G}(1,n+1)$, we base change to the field of definition $k(\ell)$ and pick coordinates $[u:v:x_1:\ldots:x_n]$ of $\mb{P}^{n+1}_{k(\ell)}$ such that $\ell=\mb{V}(x_1,\ldots,x_n)$. Now $[u:v]$ are the coordinates on $\ell$, and we have an open affine $U_\ell\subset\mb{G}(1,n+1)$ with coordinates $(a_1,b_1,\ldots,a_n,b_n)$ corresponding to lines of the (parametric) form
\[\{[u:v:a_1u+b_1v:\ldots:a_nu+b_nv]:[u,v]\in\mb{P}^1\}.\]
We obtain a local trivialization of $\Sym^{2n-1}(\mc{S}^\vee)$ by reading off the coefficients in the monomial basis $\{u^{2n-1},u^{2n-2}v,\ldots,v^{2n-1}\}$.

\begin{rem}\label{rem:base change}
Our use of base change to define coordinates may appear problematic, but no problems actually arise. The field of definition $k(\ell)$ is always a separable extension of $k$ (since we have assumed $k$ is perfect), and the local index can always be computed by ``changing base and taking trace'' over separable extensions \cite{BBMMO21}. 
\end{rem}

Our chosen coordinates and trivializations need to be compatible (in a precise sense) with a chosen \emph{relative orientation} of the bundle $\Sym^{2n-1}(\mc{S}^\vee)\to\mb{G}(1,n+1)$, which is an isomorphism
\[\det\Sym^{2n-1}(\mc{S}^\vee)\otimes\omega_{\mb{G}(1,n+1)}\cong\mc{L}^{\otimes 2}\]
for some line bundle $\mc{L}\to\mb{G}(1,n+1)$. Such an isomorphism need not exist for a general vector bundle, but it is a fact that $\Sym^{2n-1}(\mc{S}^\vee)\to\mb{G}(1,n+1)$ is relatively orientable (see e.g.~\cite[Corollary~7]{OkonekTeleman} and~\cite[Corollary~6.9]{BW23}). We will not explicitly show that our coordinates and trivialization are compatible with this relative orientation, but rather simply refer to \cite[Section 5]{cubicsurface} and \cite[Section 2.2]{PauliQuadraticTypes} for a demonstration of the proof in the $n=2$ and $n=3$ cases. The $n>3$ cases are all completely analogous.

\subsection{Computing the local index}
Because $\ell$ is a simple zero of $\sigma_F$ \cite[Theorem~6.34]{3264}, the local index $\ind_\ell\sigma_F$ can be computed by using our coordinates and trivializations to write out $\sigma_F$ as a polynomial in $(a_1,b_1,\ldots,a_n,b_n)$ and taking the Jacobian determinant of this polynomial \cite[Lemma~9]{KW19} and evaluating at $(0,0,\ldots,0,0)$.

By our choice of coordinates $[u:v:x_1:\ldots:x_n]$ on $\mb{P}^{n+1}$, we may assume that
\[F=x_1P_1(u,v)+x_2P_2(u,v)+\ldots+x_nP_n(u,v)+R(u,v,x_1,\ldots,x_n),\]
where $P_i$ are homogeneous polynomials of degree $2n-2$ and $R$ is a polynomial in the ideal $(x_1,\ldots,x_n)^2$. In our coordinates $(a_1,b_1,\ldots,a_n,b_n)$ on $U_\ell$, the section $\sigma_F$ becomes the polynomial
\[\sigma_F(\mathbf{a},\mathbf{b})=\sum_{i=1}^n(a_iu+b_iv)P_i(u,v)+R(a_1u+b_1v,\ldots,a_nu+b_nv).\]
The partial derivatives evaluated at $(0,\ldots,0)$ are now straightforward to compute:
\[\frac{\partial\sigma_F}{\partial a_i}\bigg|_{(0,\ldots,0)}=uP_i(u,v)\qquad\text{and}\qquad\frac{\partial\sigma_F}{\partial b_i}\bigg|_{(0,\ldots,0)}=vP_i(u,v).\]
To compute the Jacobian matrix, it thus suffices to write out each $P_i$ in terms of the monomial basis $\{u^{2n-1},u^{2n-2}v,\ldots,v^{2n-1}\}$, as this is our local trivialization of $\Sym^{2n-1}(\mc{S}^\vee)$. So let $p_{j,i}\in k(\ell)$ be coefficients (for $1\leq i\leq n$ and $0\leq j\leq 2n-2$) such that
\[P_i(u,v)=\sum_{j=0}^{2n-2}p_{j,i}u^jv^{2n-2-j}.\]
Then the Jacobian matrix of $\sigma_F$ at $(0,\ldots,0)$ is
\begin{equation}\label{equation: matrix of index}
A_{P_1,\cdots,P_n}:=\begin{pmatrix}
p_{2n-2,1} & 0 & p_{2n-2,2}& 0&\cdots& p_{2n-2,n}& 0\\
p_{2n-3,1} & p_{2n-2,1} & p_{2n-3,2}& p_{2n-2,2}&\cdots& p_{2n-3,n}& p_{2n-2,n}\\
\vdots & \vdots & \vdots& \vdots& \ddots& \vdots&\vdots\\
p_{0,1} & p_{1,1} & p_{0,2}& p_{1,2}&\cdots& p_{0,n}& p_{1,n}\\
0 & p_{0,1} & 0& p_{0,2}&\cdots& 0& p_{0,n} 
\end{pmatrix}.
\end{equation}
Now
\begin{equation}\label{eq:index as determinant}
    \ind_\ell\sigma_F=\Tr_{k(\ell)/k}\langle\det A_{P_1,\ldots,P_n}\rangle \in \GW(k).
\end{equation}

\subsection{Examples: lines on cubic surfaces and quintic threefolds}
We conclude this section by writing out the Jacobian matrix for the local index in the cases of lines on cubic surfaces and quintic threefolds. We will then sketch how Kass--Wickelgren and Pauli gave geometric interpretations of these determinants and what fails in the general case.

\begin{ex}[Cubic surfaces]\label{ex: cubic}
For cubic surfaces, we have $\ind_\ell\sigma_F=\Tr_{k(\ell)/k}\langle\det A_{P_1,P_2}\rangle$ with
\[A_{P_1,P_2}=\begin{pmatrix}
    p_{2,1} & 0 & p_{2,2} & 0\\
    p_{1,1} & p_{2,1}& p_{1,2} & p_{2,2}\\
    p_{0,1} & p_{1,1} & p_{0,2} & p_{1,2}\\
    0 & p_{0,1} & 0 & p_{0,2}
\end{pmatrix}.\]
Here, we have $F=x_1P_1+x_2P_2+R$ with $P_i(u,v)=p_{2,i}u^2+p_{1,i}uv+p_{0,1}v^2$. Note that $\det A_{P_1,P_2}$ is equal to the resultant of $P_1$ and $P_2$, which is where Kass and Wickelgren's geometric interpretation begins. The Gauß map along $\ell$ can be identified with the degree 2 map
\begin{align*}
    \mc{G}:\ell&\to\mb{P}^1_{k(\ell)}\\
    [u:v]&\mapsto[P_1(u,v):P_2(u,v)].
\end{align*}
In particular, for each point $p\in\ell$ there is another point $q$ such that $T_p\mb{V}(F)=T_q\mb{V}(F)$. The Segre involution $i:\ell\to\ell$ swaps $p$ and $q$. The fixed points of $i$ are defined over a quadratic field extension $k(\ell)(\sqrt{\alpha})$ for some $\alpha\in k(\ell)^\times/(k(\ell)^\times)^2$, and the Segre index of $\ell$ is $\Tr_{k(\ell)/k}\langle\alpha\rangle$.

To show that this Segre index describes $\ind_\ell\sigma_F$, we need to show that $\alpha$ and $\Res(P_1,P_2)$ agree up to squares. This is proved in \cite[Proposition~14]{cubicsurface}, but we recall the details here since we use slightly different language. The fixed points of the Segre involution are given by the ramification locus of $[P_1:P_2]$, namely the vanishing locus of the Jacobian determinant
\[\frac{\partial P_1}{\partial u}\cdot\frac{\partial P_2}{\partial v}-\frac{\partial P_2}{\partial u}\cdot\frac{\partial P_1}{\partial v}.\]
This Jacobian vanishes precisely if $F(u,v):=x_1P_2(u,v)-x_2P_1(u,v)$ has a multiple root  for some $(x_1,x_2)\neq(0,0)$. We thus calculate the discriminant $\Disc_{u,v}(x_1P_2(u,v)-x_2P_1(u,v))$, which is a quadratic function in $x_1,x_2$. The zeros of this discriminant, and hence the ramification locus of $[P_1:P_2]$, are defined over $k(\ell)(\sqrt{\alpha})$, where
\[\alpha=\Disc_{x_1,x_2}(\Disc_{u,v}(x_1P_2(u,v)-x_2P_1(u,v))).\]
We conclude by computing $16\cdot\Res(P_1,P_2)=\alpha$.
\end{ex}

For $n>2$, the determinant $\det A_{P_1,\ldots,P_n}$ is not a resultant, which explains why \cite{cubicsurface} does not generalize. However, it turns out that $\det A_{P_1,\ldots,P_n}$ is always a \emph{product} of resultants. We had wondered if this were true after Pauli treated the case of lines on quintic threefolds:

\begin{ex}[Quintic threefolds]\label{ex: quintic}
For quintic threefolds, we need to consider the determinant of
\[A_{P_1,P_2,P_3}=\begin{pmatrix}
    p_{4,1} & 0 & p_{4,2} & 0 & p_{4,3} & 0\\
    p_{3,1} & p_{4,1} & p_{3,2} & p_{4,2} & p_{3,3} & p_{4,3}\\
    p_{2,1} & p_{3,1} & p_{2,2} & p_{2,3} & p_{3,3} & p_{3,3}\\
    p_{1,1} & p_{2,1}& p_{1,2} & p_{2,2} & p_{1,3} & p_{2,3}\\
    p_{0,1} & p_{1,1} & p_{0,2} & p_{1,2} & p_{0,3} & p_{1,3}\\
    0 & p_{0,1} & 0 & p_{0,2} & 0 & p_{0,3}
\end{pmatrix}.\]
Here, we have $F=x_1P_1+x_2P_2+x_3P_3+R$ with $P_i(u,v)=p_{4,i}u^4v+p_{3,i}u^3v^2+p_{2,i}u^2v^3+p_{1,i}uv^4+p_{0,i}v^5$. Now the image of the Gauß map
\begin{align*}
    \mc{G}:\ell&\to\mb{P}^2_{k(\ell)}\\
    [u:v]&\mapsto[P_1(u,v):P_2(u,v):P_3(u,v)]
\end{align*}
is a rational plane quartic curve. A general rational plane quartic has three nodes $\nu_1,\nu_2,\nu_3$ defined over the algebraic closure $\overline{k}$. Over $k(\ell)$ we could have $1$, $2$ or $3$ nodes, in each case the sum of the degrees of the residue fields of the nodes over $k(\ell)$ equals $3$. That is, $\sum_{\text{nodes }\nu}[k(\nu):k(\ell)]=3$.

Each node $\nu$ defines a degree 2 divisor $D_{\nu}$ on $\ell_{k(\nu)}$ as follows. Consider the pencil $H^\nu_t$ of lines in $\mb{P}^2_{k(\nu)}$ through $\nu$. Then $H^\nu_t\cap\mc{G}(\ell)$ defines a pencil of degree 4 divisors $D_{\nu}+D^\nu_t$ on $\ell_{k(\nu)}$. 

By \cite[Lemma~3.1]{PauliQuadraticTypes}, the pencil $D^\nu_t$ is base point free when $\ell$ is a simple (i.e.~reduced) line and thus defines a degree 2 map of the form $[Q^\nu_1:Q^\nu_2]:\mb{P}^1_{k(\nu)}\to\mb{P}^1_{k(\nu)}$. The non-trivial element of the Galois group of the double cover $[Q^\nu_1:Q^\nu_2]$ gives an involution on $\ell_{k(\nu)}$, which we again call the \emph{Segre involution}. Geometrically, the Segre involution swaps the pairs of points in the pencil $D^\nu_t$ (see Figure~\ref{fig:n=3}). The fixed points of the Segre involution are defined over $k(\nu)(\sqrt{\alpha_\nu})$ for some $\alpha_\nu\in k(\nu)^\times/(k(\nu)^\times)^2$, which can be computed as the resultant $\alpha_\nu=\Res(Q^\nu_1,Q^\nu_2)$ by the same argument as above for $n=2$. The \emph{index} of this Segre involution is given by the norm 
\[\N_{k(\nu)/k(\ell)}\alpha_i=\N_{k(\nu)/k(\ell)}\Res(Q_1^\nu,Q_2^\nu).\]
We want to identify the local index $\ind_\ell(\sigma_F)=\Tr_{k(\ell)/k}\langle\det A_{P_1,P_2,P_3}\rangle$
with the trace of the product of Segre indices
\[\Tr_{k(\ell)/k}\langle\prod_{\text{nodes }\nu}\N_{k(\nu)/k(\ell)}\Res(Q_1^\nu,Q_2^\nu)\rangle.\]
That is, we need to show that
\[\langle\det A_{P_1,P_2,P_3}\rangle=\langle\prod_{\nu}\N_{k(\nu)/k(\ell)}\Res(Q_1^\nu,Q_2^\nu)\rangle\]
in $\GW(k(\ell))$.

In this example we assume that $\nu_1,\nu_2,\nu_3$ are all defined over $k(\ell)$. We can then use a change of coordinates to assume that
\begin{align*}
    \nu_1&=[1:0:0],\\
    \nu_2&=[0:1:0],\\
    \nu_3&=[0:0:1].
\end{align*}
Under this assumption, the quartic polynomials $P_i,P_j$ have two common zeros for $i\neq j$, and hence there exist quadratic polynomials $Q_1,Q_2,Q_3$ such that
\begin{align*}
    P_1&=Q_2Q_3,\\
    P_2&=Q_1Q_3,\\
    P_3&=Q_1Q_2.
\end{align*}
In this case, one can prove that
\begin{align*}
    [Q^1_1:Q^1_2]&=[Q_2:Q_3],\\
    [Q^2_1:Q^2_2]&=[Q_1:Q_3],\\
    [Q^3_1:Q^3_2]&=[Q_1:Q_2]
\end{align*}
and $\det A_{P_1,P_2,P_3}=\prod_{i<j}\Res(Q_i,Q_j)$, thereby proving that the local index is the product of the Segre indices.

Our use of coordinate change to write $\nu_1,\nu_2,\nu_3$ in this simple form relies on the assumption that these nodes are all defined over $k(\ell)$. This need not be the case --- in general, $\nu_1,\nu_2,\nu_3$ will be Galois conjugates over $\overline{k}$. We can then use a coordinate change to again express $P_1,P_2,P_3$ as products of quadratic polynomials $Q_1,Q_2,Q_3$, with these quadratics being Galois conjugate. See \cite{PauliQuadraticTypes} for more details.
\end{ex}

For $n>3$, we can no longer use a projective change of coordinates to reduce to a particular case as was done for quintic threefolds. This is the reason that \cite{PauliQuadraticTypes} does not admit an obvious generalization to lines on hypersurfaces of greater dimension and degree. A construction of Finashin and Kharlamov will allow us to sidestep this issue. This construction uses what we call \emph{conic models}, which we introduce in the next section.

\section{Conic models for the Gauß curve}\label{sec: types of lines 2}
Our goal in this section is to define \emph{conic models} for Gauß curves. This construction appears in \cite[$\S6$]{FinashinKharlamov2021}, and we will use it to show that the local index is the same as the product of Segre indices. In particular, this construction allows us to see that the determinant $\det A_{P_1,\cdots,P_n}$ (see Equation~\ref{equation: matrix of index}) is always a product of resultants.

\begin{rem}\label{rem:n geq 3}
    For the remainder of the article, we will assume $n\geq 3$. Most of the constructions from here on out only make sense under this assumption.
\end{rem}

We start by recalling the definition of the Segre index in general, following the same notation used in Examples \ref{ex: cubic} and \ref{ex: quintic}. This was already defined in the real case in \cite{FinashinKharlamov2021} and there are no major technical differences in passing to a general field. The main tool we will need is a result attributed to Castelnuovo, which states that there are finitely many $(n-3)$-dimensional $(2n-4)$-secants to a generic degree $2n-2$ rational curve in $\mb{P}^{n-1}$. This generalizes the fact that a general rational quartic plane curve has three nodes. Segre involutions and their associated indices are constructed from these finite sets of secants.

Finally, we will dive into conic models for Gauß curves: to each Gauß curve and its Castelnuovo secants, we can associate a set of points $B$ and a conic $Q$ in $\mathbb P^2_{k(\ell)}$. We will show that the local index and the Segre index can be computed purely in terms of a conic model.  

\subsection{Secants and the Segre index}\label{subsec: Gauss map}
Recall our setting. Let $X\subset \mathbb P^{n+1}$ be a generic hypersurface of degree $2n-1$, and fix a line $\ell\subseteq X$. Without loss of generality, we may choose coordinates $[u:v:x_1:\ldots:x_n]$ on $\mb{P}^{n+1}_{k(\ell)}$ such that $\ell = \{[u:v:x_1:\ldots:x_n]:x_1 = \cdots = x_n = 0\}$ and $X=\mb{V}(F)$ for
\[F = x_1P_1(u,v)+x_2P_2(u,v)+\ldots+x_nP_n(u,v)+R(u,v,x_1,\ldots,x_n),\]
where $P_1,\ldots,P_n\in k(\ell)[u,v]$ are homogeneous polynomials of degree $2n-2$ and $R\in(x_1,\ldots,x_n)^2\subseteq k(\ell)[u,v,x_1,\ldots,x_n]$ is also of degree $2n-2$.

Following \cite{FinashinKharlamov2021}, the Gauß map along $\ell$ is given by
\begin{align}\label{equation: gauss-map}
\begin{split}
    \gauss: \mathbb{P}^1_{k(\ell)}&\to \mathbb{P}^{n-1}_{k(\ell)} \\
    [u:v] &\mapsto [P_1(u,v):\ldots:P_n(u,v)].
\end{split}
\end{align}

\begin{defn}
We call the image of $\ell$ under the Gauß map the \emph{Gauß curve} associated to $\ell$. This is a rational curve
    \[\mc{G}(\ell)\subseteq\mb{P}^{n-1}\]
    of degree $2n-2$.
\end{defn}

We are interested in the $(n-3)$-dimensional $(2n-4)$-secants to the Gauß curve, i.e.~the $(n-3)$-planes that meet the Gauß curve in $2n-4$ points (when counted with multiplicity, i.e.~the intersection of the $(n-3)$-plane with the Gauß curve is a 0-dimensional scheme of degree $2n-4$). Finashin--Kharlamov state that there are (geometrically) $\binom{n}{2}$ such secants \cite[Proposition 4.3.3]{FinashinKharlamov2021}, which they refer to as the \emph{Castelnuovo count}.

\begin{prop}[Castelnuovo count]\label{prop: castelnuovo}
Let $k$ be a field of characteristic 0. The number of $(n-3)$-dimensional $(2n-4)$-secants (defined over $\overline{k}$) to a generic rational curve $C\subset\mb{P}^{n-1}$ of degree $2n-2$ is $\binom{n}{2}$.
\end{prop}
\begin{proof}
Just as described in the proof of \cite[Proposition 4.3.3 (2)]{FinashinKharlamov2021}, the number of secants (when counted with multiplicity) is $\binom{n}{2}$, and these multiplicities are always positive. This fact follows from the same argument present in \cite[\S 12.4.4]{3264}, where Eisenbud and Harris apply Porteous' formula to count the number of quadrisecant lines to rational curves in $\mathbb P^3$ (i.e. the case $n=4$). We explain how this argument goes in general.

Recall that the parameter space for dimension 0, degree $2n-4$ subschemes $\Gamma\subseteq\mathbb P^1$ is given by $\Sym^{2n-4}\mb{P}^1\cong\mathbb P^{2n-4}$. Consider the vector bundle $\mathcal E^*\to\mathbb P^{2n-4}$ whose fibers are given by 
\[\mathcal E^*_{\Gamma} = H^0(\mathcal O_{\mathbb P^1}(2n-2))/H^0(\mathcal I_\Gamma(2n-2)),\]
which is the space polynomials of degree $2n-2$ modulo the polynomials that vanish on $\Gamma$. Note that the rank of $\mathcal E^*$ is $2n-1 - 3 = 2n-4$.

The Gauß map $\mc{G}:\ell\to\mb{P}^{n-1}$ induces a map
\[H^0(\mathcal O_{\mathbb P^{n-1}}(1))\hookrightarrow H^0(\mathcal O_{\mathbb P^1}(2n-2)),\]
given by simply sending $x_i$ to $P_i$ (in the notation defined at the beginning of this subsection). Denoting by $\mathcal F$ the trivial rank $n$ bundle over $\mathbb P^{2n-4}$ with fiber $H^0(\mathcal{O}_{\mathbb{P}^{n-1}}(1))$, we get a morphism of vector bundles 
\[\varphi: \mathcal F\to \mathcal E^*\]
by projecting to the quotient. The $(2n-4)$-secants to $C$ of dimension $n-3$ correspond to the locus of dimension 0, degree $2n-4$ subschemes $\Gamma\subseteq C\subseteq\mathbb P^{n-1}$ whose projective linear span has codimension at least $2$. Given $\Gamma\subseteq  C\subseteq \mathbb P^{n-1}$, by the definition of the $\varphi$, the equations in the kernel of $\varphi_{\Gamma}$ correspond to hyperplanes that contain $\Gamma$ and, therefore, the linear subspace determined by the equations in $\ker \varphi$ is the projective span of $\Gamma$. To have a projective span of codimension at least 2, we need $\ker\varphi$ to have dimension at least 2 and therefore $\rank\varphi\leq n-2$. It thus suffices to consider the locus $M_{n-2}(\vphi)\subseteq\mb{P}^{2n-4}$ where the map has rank at most $n-2$.

The expected codimension of $M_{m}(\vphi)$ is $(e-m)(f-m)$, where $e$ and $f$ are the ranks of $\mathcal E^*$ and $\mathcal F$ respectively. This gives us that the expected dimension of $M_{m}(\vphi)$ is $2n-4 - (2n-4-m)(n-m)$. Notice that plugging in $m=n-2$ gives us dimension zero, which means that generic rational curves of degree $2n-2$ in $\mb{P}^{n-1}$ have a finite number of $(n-3)$-dimensional $(2n-4)$-secants. From this formula, one can also verify that the locus, where $\vphi$ has rank less than $n-2$, has negative expected dimension.
 
It remains to calculate the class $[M_{n-2}(\vphi)]$ in the Chow ring of $\mb{P}^{2n-4}$. This is done by Porteous' formula \cite[Theorem~12.4]{3264}. In the present context, Porteous' formula and the computation of the total Chern class $c(\mc{E}^*)$ \cite[Theorem~10.16]{3264} give us
\begin{align*}
[M_{n-2}(\vphi)]&=\begin{vmatrix}
c_{n-2}(\mc{E}^*) & c_{n-1}(\mc{E}^*)\\
c_{n-3}(\mc{E}^*) & c_{n-2}(\mc{E}^*)
\end{vmatrix}\\
&=\begin{vmatrix}
    \binom{n}{2}\zeta^{n-2} & \binom{n+1}{2}\zeta^{n-1}\\
    \binom{n-1}{2}\zeta^{n-3} & \binom{n}{2}\zeta^{n-2}
\end{vmatrix}\\
&=\left(\binom{n}{2}^2-\binom{n+1}{2}\binom{n-1}{2}\right)\zeta^{2n-4},
\end{align*}
where $\zeta$ is the hyperplane class. We now conclude by noting that $\binom{n}{2}^2-\binom{n+1}{2}\binom{n-1}{2}=\binom{n}{2}$.

For our purposes, we also need each of these secants to have multiplicity $1$ for a generic $C$. Again, the proof proceeds exactly as for \cite[Proposition 4.3.3 (3)]{FinashinKharlamov2021} --- one uses Kleiman's transversality theorem \cite{Kle74} (using the assumption $\op{char}{k}=0$) to show that over $\overline{k}$, there is a dense open subset of the space of degree $2n-2$ rational curves in $\mb{P}^{n-1}$ on which all secants have multiplicity $1$, since these multiplicities arise as an intersection multiplicity in an appropriate parameter space.
\end{proof}

\begin{rem}\label{rem:sottile}
Proposition~\ref{prop: castelnuovo} is the only place where we use the assumption $\op{char}{k}=0$, as this allows us to apply Kleiman's transversality theorem. Work of Sottile \cite[Conjecture 1]{Sot03} suggests that one should be able to remove this assumption.
\end{rem}

\begin{ex} 
When $n=3$, we recover the three nodes on a rational plane quartic as secants (see Figure~\ref{fig:n=3}). Indeed, we get three 2-secants of dimension zero. These are points that intersect the curve twice, i.e.~are double points. When $n=4$, we get six $4$-secants of dimension $1$ to a rational sextic in $\mb{P}^3$ (see Figure~\ref{fig:n=4}).
\end{ex}

We now define the Segre involution associated to a secant of the Gauß curve. Let $S$ be one of the $(2n-4)$-secants to $\mc{G}(\ell)$, and let $k(S)/k(\ell)$ be its field of definition. Then $S$ defines a degree $2n-4$ divisor $D_S$ on $\ell_{k(S)}$, as $S\cap\mc{G}(\ell)$ consists of $2n-4$ geometric points (when counted with multiplicity). Now consider the pencil $H_t$ of hyperplanes in $\mb{P}^{n-1}_{k(S)}$ containing $S$. For each $t$, we get a degree $2n-2$ divisor on $\ell_{k(S)}$, as $H_t\cap\mc{G}(\ell)$ consists of $2n-2$ geometric points (when counted with multiplicity) by B\'ezout's theorem. We can write this divisor as $D_S+D_t^S$. Note that $D_t^S$ is base point free and therefore defines a degree $2$ map $[Q_1^S:Q_2^S]:\mb{P}^1_{k(S)}\rightarrow \mb{P}^1_{k(S)}$. 

\begin{defn}\label{def:segre involution}
    We define the \emph{Segre involution} associated to a secant $S$ of $\mc{G}(\ell)$ as the involution
    \[i_S:\ell_{k(S)}\to\ell_{k(S)}\]
    given by the non-trivial element of the Galois group of the double covering $[Q_1^S:Q_2^S]$.
\end{defn}

The fixed points of the Segre involution $i_S$ are defined over $k(S)(\sqrt{\alpha_S})$ for some $\alpha_S\in k(S)^\times/(k(S)^\times)^2$. In fact, we may take $\alpha_S=\Res(Q_1^S,Q_2^S)$ by the argument given in Example~\ref{ex: cubic}.

\begin{defn}\label{def:type1}
The \emph{Segre index} of $\ell$ is
\[\seg(X,\ell):=\Tr_{k(\ell)/k}\langle \prod_S \N_{k(S)/k(\ell)}\alpha_S\rangle\in \GW(k)\]
where the product goes over all the secants $S$. As in Example~\ref{ex: quintic}, each secant $S$ is a representative from its orbit of Galois conjugates whose contributions to $\seg(X,\ell)$ are encoded in the norm $\N_{k(S)/k(\ell)}\alpha_S$.
\end{defn}

\begin{ex}
    In case $k=k(\ell)=\R$ the Segre index $\seg(X,\ell)$ agrees with the Segre index from the introduction defined by Finashishin--Kharlamov (you just have to stick brackets around it). Indeed, each secant defined over $\C$ will contribute the factor $N_{\C/\R}(1)=1$ and each secant $S$ defined over $\R$ will contribute the factor $\alpha_S$.
\end{ex}

\begin{rem}
As mentioned in the introduction, we will generally conflate the Segre involution $i_S:\ell_{k(S)}\to\ell_{k(S)}$, which is an involution of a projective line $\mb{P}^1_{k(S)}$, with an involution of the Gauß curve $\mc{G}(\ell)_{k(S)}$. In order to justify this conflation, we need to verify that $i_S$ induces an involution on $\mc{G}(\ell)_{k(S)}$ whose fixed points are defined over $k(S)(\sqrt{\alpha_S})$. To define an involution on $\mc{G}(\ell)_{k(S)}$, we first note that $\mc{G}(\ell)_{k(S)}=\mc{G}(\ell_{k(S)})$. We may therefore define
\begin{align*}
i_S':\mc{G}(\ell)_{k(S)}&\to\mc{G}(\ell)_{k(S)}\\
\mc{G}(x)&\mapsto\mc{G}(i_S(x)).
\end{align*}
That is, $i_S'=\mc{G}\circ i_S$. Note that $i_S(x)=x$ only if $x$ and $i_S(x)$ have the same image under the Gauß map --- this is by definition of the involution $i_S$. It follows that the fixed points of $i_S'$ are given by $\{\mc{G}(x):i_S(x)=x\}$. Finally, since $\mc{G}:\ell_{k(S)}\to\mc{G}(\ell)_{k(S)}$ is a birational map, we have an isomorphism from the field of definition of the fixed locus of $i_S$ to the field of definition of the fixed locus of $i_S'$, as desired.

Alternatively, as pointed out to us by a referee, the Gauß map is an isomorphism $\ell\to\mc{G}(\ell)$ whenever $\mc{G}(\ell)$ is smooth, which is the case for a general hypersurface. In this case, the involution $i_S$ immediately induces the desired involution $i_S'$.
\end{rem}

\subsection{Conic models}\label{sec:conic models}
Our next goal is to recap Finashin and Kharlamov's construction of \emph{conic models} for the Gauß curve \cite[Section 6.3]{FinashinKharlamov2021}. These consist of a plane conic $Q$, together with a collection a zero dimensional subscheme $B\subseteq\mb{P}^2$ of degree $\binom{n}{2}$, such that the Gauß curve coincides with the strict transform of $Q$ in the blowup $\op{Bl}_B{\mb{P}^2}$. 

One place where we deviate from Finashin--Kharlamov is that we require our plane conic $Q$ to come equipped with a parameterization, which we need in order to define an index associated to the conic. Such a parameterization comes from stereographic projection from a $k$-rational point on $Q$, but over non-closed fields there are conics defined over $k$ with no $k$-points. Working with the space of \emph{parameterized conics} allows us to sidestep this issue.

\begin{rem}
    Technically, Finashin and Kharlamov implicitly assume that their conics are parameterized, as they require their conics to have real points in order to describe the associated conic index. However, their comment on this construction is brief enough that they do not justify why one can always take a parameterized conic model. We will give such a justification in the course of this section.
\end{rem}

Let $B$ be a zero dimensional subscheme of $\P^2$ of degree $\binom{n}{2}$ of $\binom{n}2$ (geometric) points in $\mathbb P^2$ in general position, and let $Q\subseteq\mb{P}^2$ be a conic that does not pass through any of the points in $B$. Curves of degree $n-1$ through $B$ form a linear system $L_B$ of dimension
\begin{align*}
\dim L_B &= \binom{3+n-1-1}{n-1}-\binom{n}{2}-1\\
&=\binom{n+1}{2}-\binom{n}{2}-1\\
&=n-1.
\end{align*}

Therefore, after choosing a basis $\beta$ for this system, we obtain a rational map 
\[g_{B,\beta}:\mathbb P^2\dashrightarrow \mathbb P^{n-1}\]
whose indeterminacy locus is $B$. Of course, such a map depends on the choice of a basis. Finally, the curve $C := g_{B,\beta}(Q)\subset \mathbb P^{n-1}$ is a rational curve of degree $2n-2$.

We can summarize the construction $g_{B,\beta}(Q)$ in terms of a rational map to the space of rational curves of degree $2n-2$. Let $\Mor(\mb{P}^1,\mb{P}^N)_d$ denote the space of rational curves of degree $d$ in $\mb{P}^N$. Let $\mc{Q}:=\Mor(\mb{P}^1,\mb{P}^2)_2$ denote the space of parameterized conics in $\mb{P}^2$, and let $\Conf_m(\mb{P}^2)$ denote the configuration space of $m$ points in $\mb{P}^2$. The space of bases for a linear system of dimension $n-1$ is an open subscheme $\mc{U}\subseteq\mb{P}^{n^2-1}$. Altogether, Finashin--Kharlamov's construction is a rational map of the form
\begin{align*}
\FK:\mr{Conf}_{\binom{n}{2}}(\mb{P}^2)\times\mc{U}\times\mc{Q}&\dashrightarrow\Mor(\mb{P}^1,\mb{P}^{n-1})_{2n-2}\\
(B,\beta,Q)&\mapsto g_{B,\beta}(Q).
\end{align*}

\begin{defn}
    A \emph{conic model} for a rational curve $C\in\Mor(\mb{P}^1,\mb{P}^{n-1})_{2n-2}$ is an element of the fiber $\FK^{-1}(C)$, where this fiber is taken in the open subscheme in $\Conf_{\binom{n}{2}}(\mb{P}^2)\times\mc{U}\times\mc{Q}$ on which $\FK$ is a morphism. We say that $C$ \emph{has a conic model} if the fiber $\FK^{-1}(C)$ is not empty.
\end{defn}

As described in \cite[p.~4077]{FinashinKharlamov2021}, there is a close relationship between a rational curve with its secants and a conic model for the curve. We now outline this relationship, which is summarized in Table~\ref{tab:dictionary}.

As previously described, our Gauß curve $\mc{G}(\ell)$ is the image of the plane conic $Q$ in a chosen conic model $(B,\beta,Q)$. (The fact that a generic Gauß curve has a conic model will be proved in Lemma~\ref{lem:generic}.) Let $b \in B(\overline{k})$ with residue field $k(b)$. By Cramer's theorem on algebraic curves (which states that $\frac{d(d+3)}{2}$ points in general position in the plane determine a unique plane curve), there exists a unique plane curve $Z_b$ of degree $n-2$ passing through all points of $B(\overline{k})-\{b\}$. The field of definition of $Z_b$ is $k(b)$. Indeed, $B(\overline{k})-\{b\}$ is defined over $k(b)$, since both $b$ and $B$ are defined over $k(b)$. Thus $B(\overline{k})-\{b\}$ and hence $Z_b$ are fixed under $\operatorname{Gal}(k(b))$-action, giving us that $Z_b$ is defined over $k(b)$.

Bézout's theorem implies that $Z_b$ and $Q$ intersect in $2n-4$ points. Consequently, the projective linear span of $g_{B,\beta}(Z_b)$ forms an $(n-3)$-dimensional $(2n-4)$-secant to $\mathcal{G}(\ell)$. (To see this, note that the images under $g_{B,\beta}$ of $Z_b$ and a line through $b$ span a hyperplane in $\mb{P}^{n-1}$. Taking another such line, we get two hyperplanes containing $Z_b$, and their intersection is the $(n-3)$-dimensional secant we were looking for.) By verifying that distinct points $b, b' \in B(\overline{k})$ correspond to distinct $(n-3)$-planes, we conclude that all $\binom{n}{2}$ such secants to $\mathcal{G}(\ell)$ arise in this manner.  

In particular, this construction establishes a bijection between $B(\overline{k})$ and the set of $(n-3)$-dimensional geometric $(2n-4)$-secants to $\mathcal{G}(\ell)$. This bijection maps Galois conjugate points in $B(\overline{k})$ to Galois conjugate secants, while preserving the associated Galois action. This implies that a point $b \in B$ corresponds to a secant $S$ satisfying $k(b)\cong k(S)$.

Next, take the pencil $H^b_t$ of lines in $\mb{P}^2$ through a fixed $b\in B$. Then $H^b_t\cup Z_b$ is a pencil of degree $n-1$ curves through $B$, each of which consists of a fixed component $Z_b$ and a moving linear component. The projective linear span of $g_{B,\beta}(H^b_t\cup Z_b)$ is a pencil of $(n-2)$-dimensional hyperplanes in $\mb{P}^{n-1}$, each of which contains the $(2n-4)$-secant corresponding to $b$. This pencil of $(n-2)$-dimensional hyperplanes is used to construct the Segre involution
\[i_S:\mc{G}(\ell)\to\mc{G}(\ell).\]
Pulling back $i_S$ under $g_{B,\beta}$ gives an involution
\begin{equation}\label{eq:mu_b}\mu_b:Q\to Q.\end{equation}
Every involution on a conic can be obtained by intersecting the conic with a pencil of lines through some point, known as the \emph{polar point}, and swapping the pairs of points for each line. One can check that $\mu_b$ is the involution given by the pencil of lines through $b$, so that $b$ is the polar point of this involution.

\begin{table}[h]
    \centering
    \begin{tabular}{ll}
        Gauß curves & Conic models\\
        \hline
        $\mc{G}(\ell)$ & $Q$  \\
        Secant $S$ & $b\in B$\\
        Hyperplanes containing $S$ & Lines through $b$\\
        Segre involution $i_S$ & Conic involution $\mu_b$
    \end{tabular}
    \caption{Dictionary for Gauß curves and their conic models}\label{tab:dictionary}
\end{table}

We will now show that a generic rational curve of degree $2n-2$ in $\mb{P}^{n-1}$ has a conic model, as claimed in \cite[Section 6.3]{FinashinKharlamov2021}.

\begin{lemma}\label{lem:generic}
Let $C$ be a general rational curve of degree $2n-2$ in $\mathbb P^{n-1}$. Then there exists $(B,\beta,Q)\in\Conf_{\binom{n}{2}}(\mb{P}^2)\times\mc{U}\times\mc{Q}$ such that $g_{B,\beta}(Q) = C$.
\end{lemma}
\begin{proof}
The space $\Mor(\mb{P}^1,\mb{P}^N)_d$ is given by $U/\mr{PGL}_2$, where $U\subset\mb{P}(H^0(\mb{P}^1,\mc{O}(d))^{\oplus N+1})$. One can prove that $\Mor(\mb{P}^1,\mb{P}^N)_d$ is an irreducible scheme of dimension $(d+1)(N+1)-4$. To compute this dimension, note that the space of homogeneous degree $d$ polynomials in 2 variables is $d+1$, and we need $N+1$ such polynomials to define a morphism to $\mb{P}^N$. We then subtract 1 to account for projectivization, and we subtract 3 to account for the action of $\mr{Aut}(\mb{P}^1)\cong\mr{PGL}_2$. So for $N=n-1$ and $d=2n-2$, the space $\Mor(\mb{P}^1,\mb{P}^N)_d$ has dimension $2n^2-n-4$. Note that $\dim\mc{Q}=5$ and $\dim\mr{Conf}_{\binom{n}{2}}(\mb{P}^2)=2\binom{n}{2}=n^2-n$. 

The source and target of $\FK$ are both geometrically irreducible (as products of geometrically irreducible schemes). As previously computed, the source has dimension $(n^2-n)+5+(n^2-1)=2n^2-n+4$, while the target has dimension $2n^2-n-4$. To prove the desired claim, it remains to show that generic fibers of $\FK$ have dimension
\[(2n^2-n+4)-(2n^2-n-4)=8.\]
To this end, we will show that $g_{B_1,\beta_1}(Q_1)=g_{B_2,M\beta_2}(Q_2)$ for some change of basis matrix $M$ if and only if there exists a projective transformation of $\mb{P}^2$ transforming $(B_1,Q_1)$ to $(B_2,Q_2)$. From this, it will follow that generic fibers of $\FK$ have dimension $\dim\mr{Aut}(\mb{P}^2)=\dim\mr{PGL}_3=8$, as desired. 

To begin, assume that $(B_1,Q_1)$ and $(B_2,Q_2)$ are projectively equivalent. The curves $g_{B_j,\beta_j}(Q_j)$ are obtained by embedding the strict transform of $Q_j$ on the blowup $\mr{Bl}_{B_j}(\mb{P}^2)$, and an automorphism taking $(B_1,Q_1)$ to $(B_2,Q_2)$ gives an isomorphism of the strict transforms of $Q_j$ on $\mr{Bl}_{B_j}(Q_j)$. Now the embeddings $g_{B_j,\beta_j}(Q_j)$ are isomorphic but need not be equal (as elements of $\Mor(\mb{P}^1,\mb{P}^{n-1})_{2n-1}$), but they will differ by a projective change of coordinates on $\mb{P}^{n-1}$. Let $M\in\PGL_n$ represent this change of coordinates. Then $g_{B_1,\beta_1}(Q_1)=g_{B_2,M\beta_2}(Q_2)$, as desired.

We now explain why $g_{B_1,\beta_1}(Q_1)=g_{B_2,\beta_2}(Q_2)$ implies that $(B_1,Q_1)$ and $(B_2,Q_2)$ are projectively equivalent. As previously described, there is a natural bijection between $B_j$ and the set of $(n-3)$-dimensional $(2n-4)$-secants to the curve $g_{B_j,\beta_j}(Q_j)$. Each Segre involution $i_S$ on this curve determines an involution $\mu_b:Q_j\to Q_j$ with polar point $b\in B_j$. In particular, we can recover the set $B_j$ from the curve $g_{B_j,\beta_j}(Q_j)$ via the involutions $i_S$ --- assuming we already know $Q_j$. The polar point of an involution on $Q_j$ is determined up to projective change of coordinates, so the ambiguity in reconstructing $(B_j,Q_j)$ from $g_{B_j,\beta_j}(Q_j)$ is precisely $\Aut(\mb{P}^2)\cong\PGL_3$. In other words, if $g_{B_1,\beta_1}(Q_1)=g_{B_2,\beta_2}(Q_2)$, then there exists a projective transformation taking $(B_1,Q_1)$ to $(B_2,Q_2)$, as desired.
\end{proof}

\begin{rem}\label{rem:quadrisecants}
Lemma~\ref{lem:generic} states that a general rational curve of degree $2n-2$ in $\mb{P}^{n-1}$ has a conic model. For $n=4$, this description is explicitly related to Dye's result on sextic space curves and double sixes of cubic surfaces \cite{Dye97}, as mentioned in Footnote~\ref{fn:curves}. Indeed, blowing up $\binom{4}{2}$ points in general position in $\mb{P}^2$ yields a cubic surface. The strict transforms of the six conics through five of these points form one half of a double six. Now take any conic in $\mb{P}^2$ that does not meet any of the six points at the center of our blowup. The strict transform of this conic is a sextic curve on the cubic surface, and it meets each of the aforementioned lines in four points. Dye proves that every smooth rational space sextic and its six quadrisecants can be constructed in this manner.

Using the above, one can explicitly construct rational space sextics and their quadrisecants. All that remains is to give a parameterization
\[\mb{P}^2\dashrightarrow\mb{P}^3\]
of the cubic surface obtained by blowing up six points. The classical method for constructing such a parameterization is to pick a basis $f_1,f_2,f_3,f_4$ of the space of cubic forms interpolating the six points in $\mb{P}^2$. The resulting cubic surface is then the Zariski closure of the map
\[[f_1:\ldots:f_4]:\mb{P}^2\dashrightarrow\mb{P}^3.\]
For example, the sextic in Figure~\ref{fig:n=4} is the image of the conic $\mb{V}(\frac{x^2}{4}+\frac{y^2}{9}-z^2)$ on a Clebsch cubic surface, which is obtained by blowing up the points
\[B=\{[1:0:0],[0:1:0],[0:0:1],[1:1:1],[1:2:3],[2:-1:1]\}.\]
For our basis of cubic forms through these points, we chose
\begin{align*}
f_1&=\frac{1}{7}(3x^2z+xz^2+x^2y-5xy^2),\\
f_2&=\frac{1}{5}(2x^2z+xyz-3xy^2),\\
f_3&=\frac{1}{2}(-5xz^2+2yz^2+3x^2z),\\
f_4&=\frac{1}{2}(x^2z-3xz^2+2y^2z),
\end{align*}
which were adapted from \cite[Example~7]{CSD07}. The quadrisecant depicted in Figure~\ref{fig:n=4} is the image of the conic through $B-\{[1:2:3]\}$. 
\end{rem}

Not only does every general rational curve admit a conic model, but we can further take our conic and locus of points to have the same field of definition as the rational curve.

\begin{cor}\label{cor:rational conic model}
    Let $C$ be a general rational curve of degree $2n-2$ in $\mb{P}^{n-1}$. If $C$ is defined over a field $K/k$, then there exists a conic model $(B,\beta,Q)$ of $C$ such that $B$ and $Q$ are defined over $K$.
\end{cor}
\begin{proof}
    Let $q:\mb{P}^1\to\mb{P}^2$ be any parameterized conic defined over $K$, and denote $Q:=q(\mb{P}^1)$. We will show that there exist $(B,\beta)\in\Conf_{\binom{n}{2}}(\mb{P}^2)\times\mb{P}^{n^2-1}$ (with $B$ defined over $K$) such that $g_{B,\beta}(Q)=C$, which will give the desired claim. 
    
    Each secant $S$ to $C$ determines a Segre involution $i_S:\mb{P}^1_{k(S)}\to\mb{P}^1_{k(S)}$, so $q\circ i_S$ is an involution of $Q_{k(S)}$. Let $b(S)$ denote the polar point of $q\circ i_S$. Let $B$ denote the scheme whose underlying set is $\{b(S):S\text{ secant to }C\}$. Note that $B$ is defined over $K$. Indeed, all of the schemes involved are geometrically reduced closed subschemes of projective space, so they are defined over $K$ if and only if they are fixed under $\Gal(K^\mr{sep}/K)$-action (see e.g.~\cite[Proposition~2.4]{McK25}). The curve $C$ is defined over $K$, so it and its scheme of secants are fixed under all $\Gal(K^\mr{sep}/K)$-actions. Since we have assumed that $Q$ is defined over $K$, this conic is also fixed under $\Gal(K^\mr{sep}/K)$-actions, so the set of polar points of the form $b(S)$ must also be $\Gal(K^\mr{sep}/K)$-fixed.

    It remains to show that there exists $\beta$ such that $g_{B,\beta}(Q)=C$. Let $(B',\beta',Q')$ be a conic model for $C$. Since $Q'$ and $C$ are geometrically birational, they are birational after some extension $K'$ of $K$. Thus $K'$ is a field of definition of $Q'$. Since $Q'$ is a degree 2 curve and the characteristic of $K$ is not 2, we know that $K'/K$ is a separable extension. We therefore have a projective transformation $M$ (over $K'$) such that $MQ'=Q$. The set of involutions induced by $MB'$ and $B$ must be the same, as these are both induced by the involutions $i_S$ coming from $C$ and its secants. In particular, the polar points of the involutions induced by $MB'$ and $B$ must agree, so $MB'=B$. It follows that $(B,M\beta',Q)$ is a conic model for $C$.
\end{proof}

\begin{defn}
    Given a line $\ell\subseteq X$, we will say that a conic model $(B,\beta,Q)$ of $\mc{G}(\ell)$ is \emph{rational} if both $B$ and $Q$ are defined over $k(\ell)$. Corollary~\ref{cor:rational conic model} states that a rational conic model always exists.
\end{defn}

\subsection{Two equivalent formulas for the Segre index}
We have already defined the Segre index of a line $\ell$ in terms of the fixed points of the Segre involutions associated to the secants to the Gauß curve $\mc{G}(\ell)$. Using the dictionary between the Gauß curve and its conic models, we can define an alternative index associated to $\ell$.

\begin{defn}\label{def:conic index}
    Let $(B,\beta,Q)$ be a rational conic model for a Gauß curve $\mc{G}(\ell)$. For each $b\in B$, let $k(b)\supset k(\ell)$ be the field of definition of $b\in B$ and let $\alpha_b\in k(b)^\times/(k(b)^\times)^2$ be such that the fixed points of the involution $\mu_b:Q\to Q$ (see Equation~\ref{eq:mu_b}) are defined over $k(\sqrt{\alpha_b})$. The \emph{conic index} of $\ell$ is
    \[\con(X,\ell):=\Tr_{k(\ell)/k}\langle\prod_{b\in B}\N_{k(b)/k(\ell)}\alpha_b\rangle\in\GW(k).\]
\end{defn}

Note that the conic index does not depend on our choice of conic model for $\mc{G}(\ell)$. Indeed, any two such choices differ by a projective change of coordinates over $k(\ell)$ (as described in the proof of Lemma~\ref{lem:generic} and Corollary~\ref{cor:rational conic model}), and such a change of coordinates does not change the field of definition of the fixed points of an involution of the conic.

We will conclude this section by showing that the Segre index and the conic index are equal.

\begin{prop}\label{prop:type1=type2}
    Given a line $\ell$ on a degree $2n-1$ hypersurface $X\subseteq\mb{P}^{n+1}$, we have $\seg(X,\ell)=\con(X,\ell)$.
\end{prop}
\begin{proof}
Fix a rational conic model $(B,\beta,Q)$ for $\mc{G}(\ell)$. We have seen that if $S$ is the secant corresponding to $b\in B$, then $g_{B,\beta}$ induces a field isomorphism $\phi:k(S)\to k(b)$ that fixes $k(\ell)$. Furthermore, we have $\phi(\alpha_S)=\alpha_b$ up to squares simply because the involution in the conic model is mapped to the involution of the Gauß curve by $g_{B,\beta}$.
\end{proof}

\begin{ex}[The conic index for lines on a quintic threefold]
We look at the case $n=3$ (the quintic threefold case). In this case the map $g_{B,\beta}\colon \P^2\dashrightarrow \P^{n-1}=\P^2$ is a birational map and therefore a Cremona transformation.
Recall that the Gauß map $\gauss\,\mathcal\colon\, \P^1_{k(\ell)}\rightarrow \P^2_{k(\ell)}$ in this case is a general degree $4$ parametrized plane curve and that it has three nodes. These nodes are exactly the $\binom{n}{2}=3$ zero dimensional $2n-4$-secants. Assume for simplicity that the nodes are all defined over $k(\ell)$ and lie in general position (just like we did in Example \ref{ex: quintic}). Then, after a coordinate change we can assume that the nodes are the points $[1:0:0]$, $[0:1:0]$ and $[0:0:1]$. The Cremona transformation with base locus these three special points (that is, the standard Cremona transformation) is given by $\operatorname{Cr}:=[x_2x_3:x_1x_3:x_1x_2]\colon \P^2\dashrightarrow\P^2$.
Now $\operatorname{Cr}\circ\,\mc{G}\,\colon \P^1_{k(\ell)}\rightarrow \P^2_{k(\ell)}$ has degree $2$, so we have a parametrized conic $Q:=\operatorname{Cr}\circ\,\mc{G}$.
Let $B=\{[1:0:0],[0:1:0],[0:0:1]\}$ and $g_{B,\beta}=\operatorname{Cr}$. Then $(B,\beta,Q)$ is a conic model for $\mathcal{G}$ since $\operatorname{Cr}$ is a birational involution with base locus $B$.

We compute the conic index in this setup.
To do so, we first have to write down a parametrization of $Q=\operatorname{Cr}\circ\,\mc{G}$. Recall from Example \ref{ex: quintic} that in this case, the Gauß map $\mathcal{G}$ is given by $\mathcal{G}=[Q_2Q_3:Q_1Q_3:Q_1Q_2]$ and thus $\operatorname{Cr}\circ\,\mc{G}=[Q_1:Q_2:Q_3]$.
For the polar point $[1:0:0]\in B$ the pencil of degree two divisors on $\P^1_{k(\ell)}$ is given by $\{t_0Q_2+t_1Q_3=0\}\subset \P^1_{k(\ell)}$ and thus the fixed points of the associated involution live over $k(\ell)(\sqrt{\alpha})$ with $\alpha=\Res(Q_2,Q_3)$ by the same argument as in Example \ref{ex: cubic}. Similarly, the fixed points of the involutions for the other two points in $B$ live over $k(\ell)(\sqrt{\alpha})$ with $\alpha=\Res(Q_1,Q_3)$ respectively $\alpha=\Res(Q_1,Q_2)$. Thus the conic index equals 
\[\Res(Q_2,Q_3)\Res(Q_1,Q_3)\Res(Q_1,Q_2),\]
which agrees with the Segre index and local index by Example \ref{ex: quintic}.
\end{ex}

\section{The local index is the conic index}\label{sec:local = segre}

We are almost ready to prove Theorem~\ref{thm:local index}, which states that the local index of a line on a degree $2n-1$ hypersurface in $\mb{P}^{n+1}$ is given by the Segre index. By Proposition~\ref{prop:type1=type2}, it suffices to show that the local index is equal to the conic index, which is the following theorem.

\begin{thm}\label{thm:local index = conic index}
    If $\ell$ is a line on a degree $2n-1$ hypersurface $X=\mb{V}(F)$ in $\mb{P}^{n+1}$, then $\ind_\ell\sigma_F=\con(X,\ell)$.
\end{thm}
\begin{proof}
For any $X$ and $\ell$, there exists a rational conic model $(B,\beta,Q)$ of $\mc{G}(\ell)$ according to Lemma \ref{lem:generic} and Corollary \ref{cor:rational conic model}. We will prove that $\con(X,\ell)$ can be expressed as a product of resultants in terms of $Q$ and $B$, denoted by $\mc{R}(B,Q)$ (see Equation~\ref{equation: prod-res}). Using the coordinate functions of our rational curve $\mc{G}(\ell)=g_{B,\beta}(Q)$, we will construct a matrix whose determinant only depends on $B$ and $Q$. We will denote the determinant of this matrix by $\mc{A}(B,Q)$ (Equation~\ref{eq:defining ABQ}).

So far, $\mc{R}(B,Q)$ and $\mc{A}(B,Q)$ are algebraic maps on $\Conf_{\binom{n}{2}}(\mb{P}^2)\times\mc{Q}$. By applying a projective change of coordinates if necessary, we may assume that $B\subseteq\mb{P}^2$ does not intersect the divisor at infinity, so that $\mc{R}(B,Q)$ and $\mc{A}(B,Q)$ are algebraic maps on $\Conf_{\binom{n}{2}}(\mb{A}^2)\times\mc{Q}$. As $\mc{Q}$ is the space of parameterized conics, we have a parameterization $[Q_0:Q_1:Q_2]$ of $Q$. The space of coefficients of a homogeneous degree 2 polynomial in 3 variables is $\mb{A}^3$, so we may regard $(Q_0,Q_1,Q_2)$ as a point in $\mb{A}^9$. As 
\[\Conf_{\binom{n}{2}}(\mb{A}^2)\subseteq\mb{A}^{2\binom{n}{2}},\]
we may treat $\mc{R}(B,Q)$ and $\mc{A}(B,Q)$ as algebraic maps on $\mb{A}^{n(n-1)}\times\mb{A}^9$. In order for this change of domain to be well-defined, it suffices to prove that the maps $\mc{R}(B,Q)$ and $\mc{A}(B,Q)$ do not depend (up to squares) on:
\begin{enumerate}
    \item the choice of coordinates on $\mb{P}^2$,
    \item the choice of parameterization of $Q$, and
    \item the choice of representative of the projective equivalence class $[Q_0:Q_1:Q_2]$.
\end{enumerate}
This will be proved in Propositions~\ref{prop:independence for t} and~\ref{prop:independence for a}. We will then show that
\begin{align*}
    \mc{R}(B,Q)&=\prod_{b\in B}\N_{k(b)/k(\ell)}\alpha_b,\\
    \mc{A}(B,Q)&=\det{A_{P_1,\ldots,P_n}}
    \cdot V(B,Q),
\end{align*}
for some regular map $V(B,Q)$,
where we use the notation of Definition~\ref{def:conic index} and Equation~\ref{equation: matrix of index}, respectively. If we can prove that $\mc{R}(B,Q)\cdot V(B,Q)=\mc{A}(B,Q)$ whenever $\mc{A}(B,Q)\neq 0$, then the claim that
\[\con(X,\ell)=\ind_\ell\sigma_F\]
will follow from Definition~\ref{def:conic index} and Equation~\ref{eq:index as determinant}.

To prove that $\mc{R}(B,Q)\cdot V(B,Q)=\mc{A}(B,Q)$ whenever $\mc{A}(B,Q)\neq 0$, we will show: 
\begin{enumerate}[(i)]
    \item if $V(B,Q)\neq 0$, then $\mc{A}(B,Q)\neq 0$ (Lemma~\ref{lem: det(VB) =0 implies ABQ=0}),
    \item $V(B,Q)\cdot\mc{R}(B,Q)$ and $\mc{A}(B,Q)$ have the same zero locus (Proposition \ref{prop: same zero locus}), 
    \item there exist infinitely many $(B,Q)$ (over $\overline{k(\ell)}$) such that $V(B,Q)\cdot\mc{R}(B,Q)=\mc{A}(B,Q)\neq 0$ (Proposition \ref{prop: agree at somethint nonzero}).
\end{enumerate}
We then conclude that $\mc{A}(B,Q)=V(B,Q)\cdot\mc{R}(B,Q)$ in Corollary~\ref{cor:equal}.
\end{proof}

\begin{rem}
    Note that our choices of $(B,Q)$ in item (iii) need not satisfy any genericity conditions beyond $\mc{A}(B,Q)\neq 0$.
\end{rem}

\begin{proof}[Proof of Theorem~\ref{thm:local index}]
By Proposition~\ref{prop:type1=type2}, the Segre index equals the conic index. By Theorem~\ref{thm:local index = conic index}, the conic index equals the local degree. This means that the local index equals the Segre index, as claimed.
\end{proof}

The rest of the paper is devoted to proving the propositions and lemmas referenced in the proof of Theorem~\ref{thm:local index = conic index}. For the remainder of this section fix a rational conic model $(B,\beta,Q)$ of $\mc{G}(\ell)$. Let $[Z:X:Y]$ be coordinates on $\mb{P}^2$ such that $B\subseteq\{Z\neq 0\}$. We denote the affine coordinates on $\{Z\neq 0\}$ by $(x,y)$. Let $[Q_0:Q_1:Q_2]$ be a parameterization of $Q$ (which exists, as $Q$ is a parameterized conic).

\subsection{Defining $\mc{R}(B,Q)$}
Recall that 
\[\con(X,\ell)=\Tr_{k(\ell)/k}\langle \prod_{b \in B} \N_{k(b)/k(\ell)} \alpha_b \rangle.\]  
Our goal is to express $\prod_{b \in B} \N_{k(b)/k(\ell)} \alpha_b$ in terms of $B$ and $Q$, where $\alpha_b\in k(b)^\times$ is such that the fixed points of the involution $\mu_b$ are defined over $k(b)(\sqrt{\alpha_b})$. Let $b=(b_x,b_y)\in \A^2=\{Z\neq 0\}\subset \P^2$. We have 
\[
\alpha_b=\Disc_{t_0, t_1}\big(\Disc_{u, v}(t_0(Q_1(u, v) - b_x Q_0(u, v)) - t_1(Q_2(u, v) - b_y Q_0(u, v)))\big),
\]
which simplifies to  
\[
\Res(Q_1 - b_x Q_0, Q_2 - b_y Q_0)
\]  up to the factor $16\in k(\ell)^2$.
We can now define $\mc{R}(B,Q)$ as  
\begin{equation}\label{equation: prod-res}
    \mc{R}(B,Q) = \prod_{b\in B(\overline{k(\ell)})} \Res(Q_1 - b_x Q_0, Q_2 - b_y Q_0).
\end{equation}

 Note that the product defining $\mc{R}(B,Q)$ runs over the set of geometric points of $B$. Each closed point $b\in B$ consists of a Galois orbit of geometric points, and the product over such an orbit gives the field norm $\N_{k(b)/k(\ell)}$. Thus, we have
\[\mc{R}(B,Q)=\prod_{b\in B}\N_{k(b)/k(\ell)}\alpha_b\]
up to squares in $k(\ell)$.
 In particular, if $B$ and $Q$ are defined over $k(\ell)$, then so is $\mc{R}(B,Q)$.

\begin{prop}\label{prop:independence for t}
    As an element of $k(\ell)^\times/(k(\ell)^\times)^2$, the value $\mc{R}(B,Q)$ does not depend on:
    \begin{enumerate}[(i)]
    \item the choice of coordinates on $\mb{P}^2$,
    \item the choice of parameterization of $Q$, and
    \item the choice of representative of the projective equivalence class $[Q_0:Q_1:Q_2]$.
\end{enumerate}
\end{prop}
\begin{proof}
All three of these statements can be verified computationally, as we now explain.
\begin{enumerate}[(i)]
    \item Any change of coordinates on $\mb{P}^2$ can be represented by (the projective class of) some $(a_{ij})=A\in \operatorname{GL}_3(k(\ell))$. One can compute directly (with your favorite computer algebra system) that after such a coordinate change, we get
    \begin{equation}\label{eq:resultant by square}
        (a_{11}+a_{12}b_x+a_{13}b_y)^2\cdot (\det A)^2\cdot \Res(Q_2-b_xQ_1,Q_3-b_yQ_1),
    \end{equation}
    which differs from $\Res(Q_2-b_xQ_1,Q_3-b_yQ_1)$ by a square. Note that even if $(b_x,b_y)$ is not defined over $k(\ell)$, the value $\mc{R}(B,Q)$ is a product over all $b\in B(\overline{k})$. Our assumption that $B$ is defined over $k(\ell)$ implies that the Galois conjugates of Equation~\ref{eq:resultant by square} will also by factors in this product, so that $\mc{R}(B,Q)$ will only change by a square in $k(\ell)^\times$ after our change of coordinates $A$.
    \item Choosing a different parametrization of $Q$ is simply precomposing with an automorphism of $\mathbb P^1$. It suffices to show that a resultant $\Res(A(z),B(z))$ changes by a square after Möbius transformations when $\deg A = \deg B = 2$. Applying Möbius transformations to $A$ and $B$ in this case is equivalent to taking new polynomials
    \[A'(z) = (cz+d)^2A\left(\frac{az+b}{cz+d}\right),
\qquad
    B'(z) = (cz+d)^2A\left(\frac{az+b}{cz+d}\right).\]

    Classical properties of resultants imply that
    \begin{itemize}
        \item $\Res(A(z+a),B(z+b)) = \Res(A(z),B(z))$,
        \item $\Res(A(az),B(az)) = a^{\deg A\cdot\deg B}\Res(A(z),B(z))=a^4\Res(A,B)$, and
        \item $\Res(z^{\deg A}A(1/z),z^{\deg B}B(1/z)) = (-1)^{\deg A\cdot\deg B}\Res(A(z),B(z)) = \Res(A,B)$.
   \end{itemize}

    Since Möbius transformations are compositions of translations, invertions and scalar multiplication (which are invariant up to squares), we get the result. In particular, by computing directly, we can see that
    \[\Res(A',B') = \det\left(\begin{matrix}a &b\\c & d\end{matrix}\right)^4\Res(A,B).\]

    \item For a different choice of projective class, we have $Q_i'=\lambda Q_i$ for some $\lambda\in k(\ell)^{\times}$. Again, it suffices to see that $\Res(\lambda A,\lambda B) = \Res(A,B)$ up to squares. Indeed, we are simply multiplying the Sylvester matrix by $\lambda$ and therefore
    \[\Res(\lambda A,\lambda B) = \lambda^{\deg A+\deg B}\Res(A,B).\]
    In our case this implies that the difference will given by a factor $\lambda ^4$, which is a square.\qedhere
\end{enumerate}
\end{proof}

\subsection{Defining $\mc{A}(B,Q)$}
The definition of the regular function $\mc{A}(B,Q)$ is more involved. The strategy is to construct a basis $\beta'$ in terms of $(B,Q)$. This will allow us to construct a rational curve
\[g_{B,\beta'}\circ[Q_0:Q_1:Q_2]:\mb{P}^1\to\mb{P}^{n-1}.\]
We will then define the regular function $\mc{A}(B,Q)$ as the determinant of the matrix $A_{P_1,\ldots,P_n}$ (Equation~\ref{equation: matrix of index}), where $P_1,\ldots,P_n$ are the coordinate functions of $g_{B,\beta'}\circ[Q_0:Q_1:Q_2]$.

We will now describe how to construct our basis $\beta'$. Let $B(\overline{k})=\{b_1,\ldots,b_m\}$, where $m=\binom{n}{2}$. In our chosen affine patch, write $b_i=(b_{i,x},b_{i,y})$. Consider the following interpolation matrix:
\begin{equation}\label{equation: matrixN}
    \VB = \begin{pmatrix}
        1 &b_{1,x}& b_{1,y}& b_{1,x}^2 & b_{1,x}b_{1,y} &\cdots& b_{1,y}^{n-2}\\
        1 &b_{2,x}& b_{2,y}& b_{2,x}^2 & b_{2,x}b_{2,y} &\cdots& b_{2,y}^{n-2}\\
        \vdots & \vdots & \vdots & \vdots & \vdots & \ddots & \vdots\\
        1 &b_{m,x}& b_{m,y}& b_{m,x}^2 & b_{m,x}b_{m,y} &\cdots& b_{m,y}^{n-2}\\
      \end{pmatrix}.
\end{equation}

Note that $\VB$ is a square $m\times m$ matrix, since the number of monomials in two variables of degree at most $n-2$ in two variables is simply $\sum_{i=1}^{n-1}i=\binom{n}{2}=m$. We can think of $\VB$ as a matrix in the space of polynomials in two variables of degree at most $n-2$. A solution of the system $\VB f = w$ is a polynomial $f$ for which $f(b_{i,x},b_{i,y}) = w_i$ for all $b_i\in B$. If $\det{\VB} = 0$, then there is a curve of degree $n-2$ through all points of $B$. In this sense, $\VB$ is a generalization of the Vandermonde matrix (as is also the case for more general interpolation matrices).

\begin{lemma}
\label{lemma:fieldofdefofdetVB}
    If $B$ is defined over $k(\ell)$, then $\det \VB$ is defined over a quadratic field extension of $k(\ell)$. Furthermore, $(\det \VB)^2\in k(\ell)$.
\end{lemma}

\begin{proof}

There is a group homomorphism
\[\rho\colon \operatorname{Gal}(\overline{k}/k(\ell))\rightarrow S_{m}\]
where \(\rho(\sigma)\) is the permutation induced by \(\sigma\) on the rows of \(\VB\), for each
\(\sigma \in \operatorname{Gal}(\overline{k}/k(\ell))\). Post-composing with the sign map, we get
\[
\chi = \operatorname{sgn} \circ \rho : \mathrm{Gal}(\overline{k}/k(\ell)) \to \{\pm 1\}.
\]
Let \(H = \ker(\chi)\). 
Then for $\sigma \in H$ we have that
\[\sigma(\det\VB)=\chi(\sigma)\det \VB=\det\VB.\]
In particular, we have \(\det \VB \in \overline{k}^H\) is in the subfield of $\overline{k}$ fixed by $H$. Because \(\chi\) has image in \(\{\pm1\}\), the index
\([\mathrm{Gal}(\overline{k}/k):H]\le 2\). Thus \(\overline{k}^H/k(\ell)\) is at most quadratic.

Since $(\det \VB)^2=\det (\VB\otimes \mathtt{I}_2)$ where $\VB\otimes \mathtt{I}_2$ is the Kronecker product of $\VB$ with the $2\times 2$ identity matrix, and $\chi(\VB\otimes \mathtt{I}_2)=1$, we get $(\det\VB)^2\in k(\ell)$.
\end{proof}

In order to get the linear system $L_B$ of curves of degree $n-1$ through the points $B$, we need to add the following $n$ columns to $\VB$:
\begin{equation}\label{equation: matrix RB}
   \RB =  \begin{pmatrix} 
    b_{1,x}^{n-1} & b_{1,x}^{n-2}b_{1,y} & \cdots & b_{1,y}^{n-1}\\
    b_{2,x}^{n-1} & b_{2,x}^{n-2}b_{2,y} & \cdots & b_{2,y}^{n-1}\\
    \vdots & \vdots & \ddots & \vdots \\
    b_{m,x}^{n-1} & b_{m,x}^{n-2}b_{m,y} & \cdots & b_{m,y}^{n-1}
    \end{pmatrix}.
\end{equation}

Appending $\RB$ to $\VB$ yields the $m\times (m+n)$-matrix 
\[\LB = \left[\VB\mid\RB\right].\]

By construction, elements of the kernel of $\LB$ are exactly the elements of the linear system $L_B$. Our next goal is to construct $\beta'$ as a particular basis of $\ker{\LB}$. To begin, let $\VB^*$ be the adjugate (i.e.~cofactor transpose) of $\VB$ and define a block matrix
\begin{equation}\label{equation: matrix KB}
    \KB := \begin{pmatrix}
        -\VB^*\cdot\RB \\
        \det{\VB}\cdot \mathtt{I}_n
    \end{pmatrix},
\end{equation}
where $\mathtt{I}_n$ denotes the $n\times n$ identity matrix. Note that $\LB\cdot \KB = 0$. If $\det\VB\neq 0$, the columns of $\KB$ are linearly independent, in which case the columns of $\KB$ form a basis for $\ker{\LB}$. Let $\beta'$ be such a basis.

Finally, let $\mathtt{Q}$ be the $(2n+1)\times(m+n)$ matrix representing the linear map
\begin{align*}
    k(\ell)[X,Y,Z]_{(n-1)}&\to k(\ell)[u,v]_{(2n-2)}\\
    f(x,y,z)&\mapsto f(Q_0,Q_1,Q_2).
\end{align*}
Then $\mathtt{Q}\cdot\KB$ is a matrix whose columns give the coefficients of the $n$ coordinate polynomials of the parameterized curve $C=g_{B,\beta'}(Q)\subseteq\mb{P}^{n-1}$. Denote these polynomials by $P_1',\ldots,P_n'$, and let 
\begin{equation}\label{eq:defining ABQ}
\mc{A}(B,Q):=\det{A_{P_1',\ldots,P_n'}},
\end{equation}
where $A_{P_1',\ldots,P_n'}$ is the matrix given in Equation~\ref{equation: matrix of index}.

Ideally, we would like to replace \(\VB^*\) by \(\VB^{-1}\), since this would eliminate the factor
\((\det \VB)^{2n}\) in the expression for \(\mathcal{A}(B,Q)\).  
However, this requires \(\VB\) to be invertible, which need not hold in general.

In case \(\VB\) is invertible, we have
\[
\VB^* = \det(\VB)\,\VB^{-1}
\]
and can define
\[
\KB' \coloneqq \frac{1}{\det \VB}\,\KB
= \begin{pmatrix}
-\VB^{-1}\RB \\
\mathtt{I}_n
\end{pmatrix}.
\]

In this case, we may use \(\KB'\) in place of \(\KB\) to obtain a basis \(\beta''\) and associated
polynomials \(P_1'',\ldots,P_n''\).
A direct computation then shows that
\begin{equation}
\label{eq:AP1''...Pn''}
\mathcal{A}(B,Q)
= (\det \VB)^{2n}\cdot \det A_{P_1'',\ldots,P_n''}.
\end{equation}
We define
\[V(B,Q)\coloneqq (\det \VB)^{2n}.\]
Note that $V(B,Q)$ is a morphism (rather than just a rational map), since it is well-defined even when $\det \VB=0$.

\begin{prop}
\label{prop:fieldofdefP''}
If $\det V_B\ne0$, $\beta''$ is a basis for the linear system $L_B$ over $k(\ell)$ and the polynomials $P_1'',\dots, P_n''$ have coefficients in $k(\ell)$.
\end{prop}
\begin{proof}
    Since the lower part of the matrix $K_B'$ is the identity, it is clear that the columns forming $\beta$ are all linearly independent. It remains to prove that the entries in $\VB^{-1}\RB$ are all in $k(\ell)$. 

    Since $B$ is defined over $k(\ell)$, the points $b_1,\dots,b_n$ form a closed set under the Galois action, i.e, for each $\sigma \in {\rm Gal}(\bar k/k(\ell))$, we have that $\sigma(b_i) = (\sigma(b_{i,x}),\sigma(b_{i,y})) =(b_{j,x},b_{j,y}) = b_j$ for some $j$ (which can be equal to $i$). This implies that the action of any element of the Galois group corresponds to multiplication by a permutation matrix $P_{\sigma}$. 
    
     The rows of the matrices $\VB$ and $\RB$ are monomials in $b_{i,x}$ and $b_{i,y}$. The action of $\sigma\in{\rm Gal}(\bar k/k)$, therefore, just permutes the rows. 
     \begin{align*}
    \sigma(\VB) &= P_{\sigma}\VB\\
     \sigma(\RB) &= P_{\sigma} \RB
     \end{align*}
     Since $\sigma(\VB^{-1}) = (\sigma(\VB))^{-1} = \VB^{-1}P_{\sigma}^{-1}$, we have:
     $$\sigma(\VB^{-1}\RB) = \sigma(\VB^{-1})\sigma(\RB)= \VB^{-1}P_{\sigma}^{-1}P_\sigma\RB = \VB^{-1}\RB$$

     This implies that the entries of $\VB^{-1}\RB$ are all in $k(\ell)$ as we wanted.
\end{proof}

Note that in contrast to $P_1'',\ldots,P_n''$, the coefficients of $P_1',\ldots,P_n'$ might not have coefficients in $k(\ell)$, but rather in some quadratic field extension (see Lemma \ref{lemma:fieldofdefofdetVB}).

Additionally, if $\det\VB=0$, then $\beta'$ is not necessarily a basis. However, we have $\det\VB\neq 0$ for a general hypersurface in $\mb{P}^{n+1}$. Indeed, if $\det\VB=0$, then there  exists a degree $n-2$ curve passing through $B$. The projective span of this curve is a $(2n-4)$-secant of dimension $n-3$ to the Gauss curve corresponding to our hypersurface. As the codimension of this secant is 2, we have infinitely many $(2n-4)$-secants of codimension 1, which does not occur for general hypersurfaces of degree $2n-1$ in $\mb{P}^{n+1}$.

Nevertheless, the matrix $\mathtt{Q}\cdot \mathtt{K_B}$ can still be computed even when $\det\VB=0$ and, therefore, the map $\mc{A}(B,Q)$ can still be defined in this case. In fact, we will prove in Lemma~\ref{lem: det(VB) =0 implies ABQ=0} that $\mc{A}(B,Q)=0$ whenever $\det{\VB}=0$. This is the main reason for considering the adjugate matrix $\VB^*$ instead of the inverse: so we can have the map defined in the whole affine space.

As previously mentioned, we need to verify that $\mc{A}(B,Q)$ and the determinant of the local index matrix of our Gauß curve $\mc{G}(\ell)$ agree up to a regular function $V(B,Q)$, that will also appear when comparing $\mc{A}$ and $\mc{R}$. This follows from the fact that they differ by a projective change of coordinates:

\begin{lemma}\label{lem:determinants up to squares}
Let $(B,\beta,Q)$ be a rational conic model of a Gauß curve $\mc{G}(\ell)$. If $\gamma$ is another basis of the linear system of degree $2n-1$ curves through $B$, then the coordinate polynomials $P_1,\ldots,P_n$ of $\mc{G}(\ell)$ and $P_1',\ldots,P_n'$ of $g_{B,\gamma}\circ Q$ differ by a projective change of coordinates, and
    \[\det{A_{P_1,\ldots,P_n}}=(\det M)^2\det{A_{P_1',\ldots,P_n'}}\]
    where $M=(m_{ij})\in\mathrm{GL}_n(\overline{k})$ is the matrix such that $\beta=M\cdot\gamma$.
\end{lemma}

\begin{proof}
    We then have $\beta_i=\sum_{j=1}^n m_{ij}\gamma_j'$, and hence composing with our conic $Q$ gives us $P_i=\sum_{j=1}^nm_{ij}P_j'$. In particular, the class of $M$ in $\mathrm{PGL}_n(k(\ell))$ gives us a projective equivalence from $[P_1:\ldots:P_n]$ to $[P_1':\ldots:P_n']$. Now let
    \[\tilde{M}=\begin{pmatrix}
        m_{11} & 0 & m_{21} & 0 & \cdots & m_{n1} & 0\\
        0 & m_{11} & 0 & m_{21} & \cdots & 0 & m_{n1}\\
        m_{12} & 0 & m_{22} & 0 & \cdots & m_{n2} & 0\\
        0 & m_{12} & 0 & m_{22} & \cdots & 0 & m_{n2}\\
        \vdots & \vdots & \vdots & \vdots & \ddots & \vdots & \vdots \\
        m_{1n} & 0 & m_{2n} & 0 & \cdots & m_{nn} & 0\\
        0 & m_{1n} & 0 & m_{2n} & \cdots & 0 & m_{nn}
    \end{pmatrix}.\]
    (In other words, $\tilde{M}=M\otimes\mathtt{I}_2$ is the Kronecker product of $M$ and the $2\times 2$ identity matrix.) Note that $A_{P_1,\ldots,P_n}=A_{P_1',\ldots,P_n'}\cdot\tilde{M}$. Moreover, we have $\det{\tilde{M}}=\det{M}^2$, so it follows that $\det{A_{P_1,\ldots,P_n}}=\det{A_{P_1',\ldots,P_n'}}$ up to squares.
\end{proof}
\begin{rem}
\label{rem:detdifferuptosquares}
 If both $\beta$ and $\gamma$ in Lemma \ref{lem:determinants up to squares} are defined over $k(\ell)$, the lemma implies that the determinants $\det A_{P_1,\dots,P_n}$ and $A_{P_1',\dots, P_n'}$ agree up to squares in the field $k(\ell)$.
\end{rem}
 
\begin{rem} 
Suppose we start with $P_1,\ldots,P_n$ with coefficients in $k(\ell)$ and find a conic model $(B,\beta, Q)$ as in Corollary \ref{cor:rational conic model}. We then continue as in this section with $(B,Q)$ to get a new basis $\beta'$. For a general choice of $P_1,\ldots,P_n$, the matrix $\VB$ is invertible, so we can replace $\KB$ by $\KB'$ and get a basis $\beta''$ and polynomials $P_1'',\ldots,P_n''$ (see discussion before Proposition \ref{prop:fieldofdefP''}). Then $\beta$ and $\beta''$ are both defined over $k(\ell)$ by Proposition \ref{prop:fieldofdefP''}, so by Lemma \ref{lem:determinants up to squares} and Remark \ref{rem:detdifferuptosquares} we get that $\det A_{P_1,\ldots,P_n}$ and $\det A_{P_1'',\ldots,P_n''}$ agree up to squares in $k(\ell)$. Thus the local index we want to find is the trace $\Tr_{k(\ell)/k}$ of
  \[\langle \det A_{P_1,\ldots,P_n}\rangle=\langle \det A_{P_1'',\ldots,P_n''}\rangle=\left\langle \frac{\det A_{P_1',\ldots,P_n'}}{V(B,Q)}\right\rangle\]
  with 
  \[V(B,Q)=(\det \VB)^{2n} \in k(\ell)\]
  by \eqref{eq:AP1''...Pn''}.
\end{rem}

Finally, we need to justify that $\mc{A}(B,Q)$ does not depend on the various choices related to picking representatives of $B$ and $Q$.

\begin{prop}\label{prop:independence for a}
As an element of $k(\ell)^\times/(k(\ell)^\times)^2$, the value $\mc{A}(B,Q)$ does not depend on:
\begin{enumerate}[(i)]
\item the choice of coordinates on $\mb{P}^2$,
\item the choice of parameterization of $Q$, and
\item the choice of representative of the projective equivalence class $[Q_0:Q_1:Q_2]$.
\end{enumerate}
\end{prop}
\begin{proof}
    As with Proposition~\ref{prop:independence for t}, this boils down to some computations.
    \begin{enumerate}[(i)]
    \item Let $\varphi\colon \P^2\rightarrow \P^2$ be a change of coordinates, and let $B'=\varphi(B)$. Then $\varphi$ induces a map $\widetilde{\varphi}\colon L_B\rightarrow L_{B'}$ of linear systems given by $C\mapsto C\circ \varphi^{-1}$. Pick $\beta$ be a basis for $L_B$. We then obtain a basis $\beta'$ for $L_{B'}$ given by $\beta'_i=\widetilde{\varphi}(\beta_i)$ for each $\beta_i\in B$ (recall from Lemma \ref{lem:determinants up to squares} that the choice of basis does not matter). Now
    \begin{align*}
        g_{B',\beta'}(\varphi(Q))&=[\beta_1'\circ\varphi(Q):\ldots:\beta_n'\circ\varphi(Q)]\\
        &=[\beta_1\circ\varphi^{-1}\circ\varphi(Q):\ldots:\beta_n\circ\varphi^{-1}\circ\varphi(Q)]\\
        &=[\beta_1(Q):\ldots:\beta_n(Q)]\\
        &=g_{B,\beta}(Q).
    \end{align*}
    \item Let $[Q_0:Q_1:Q_2]$ be a parameterization of $Q$, and let $\vphi:\mb{P}^1\to\mb{P}^1$ be a Möbius transformation. Then the coordinate polynomials of $g_{B,\beta}(Q)$ and $g_{B,\beta}(Q\circ\vphi)$ satisfy
    \[\{[P_1':\ldots:P_n']\}=\{[P_1'\circ\vphi:\ldots:P_n'\circ\vphi]\},\]
    and hence the coefficient matrices $A_{P_1',\ldots,P_n'}$ and $A_{P_1'\circ\vphi,\ldots,P_n'\circ\vphi}$ differ by $M\otimes\mathtt{I}_2$, where $M$ is some invertible $n\times n$ matrix. It follows that the determinants of these two matrices differ by $\det{M}^2$ (see the proof of Lemma \ref{lem:determinants up to squares}), which is a square.
    \item If $\lambda\in k(\ell)^\times$, then changing $(Q_0,Q_1,Q_2)$ to $(\lambda Q_0,\lambda Q_1,\lambda Q_2)$ changes the matrix $\mathtt{Q}$ by scaling each column by $\lambda^{n-1}$. As a result, the polynomials $P_1',\ldots,P_n'$ (whose coefficients are given by the columns of $\mathtt{Q}\cdot\KB$) are sent to polynomials $\lambda^{n-1} P_1',\ldots,\lambda^{n-1}P_n'$. Since $A_{P_1',\ldots,P_n'}$ has rank $2n$, it follows that
    \[\det{A_{\lambda^{n-1} P_1',\ldots,\lambda^{n-1}P_n'}}=\lambda^{2n(n-1)}\cdot\det{A_{P_1',\ldots,P_n'}},\]
    so these two determinants differ by a square.\qedhere
    \end{enumerate}
\end{proof}

\subsection{Proving $\mc{A}(B,Q)=\mc{R}(B,Q)\cdot (\det\VB)^{2n}$} 

We have now defined our regular maps $\mc{A}(B,Q)$ and $\mc{R}(B,Q)$ that compute the conic index and local index, respectively. The final step is to show that these two regular maps agree up to squares.
These are both regular maps on the affine spaces $\mb{A}^{n(n-1)}\times\mb{A}^9$ parameterizing $(B,Q)$ (see proof of Theorem \ref{thm:local index = conic index}).
Our first goal to this end is to show that $\mc{A}(B,Q)$ and $(\det\VB)^{2n}\cdot\mc{R}(B,Q)$ have the same zero locus.

\begin{lemma}\label{lem: det(VB) =0 implies ABQ=0}
    If $\det{\VB}=0$, then $\mc{A}(B,Q)=0$.
\end{lemma}
\begin{proof}
    If $\det{\VB} = 0$, then the bottom $n$ rows of $\KB$ are 0. In the polynomials corresponding to the columns of $\KB$, this means that the coefficients of all monomials that do not contain $z$ are 0. (Recall that here, $[x:y:z]$ are our projective coordinates for $\mb{P}^2$, and the columns of $\KB$ give the coefficients of corresponding plane curves.) Hence the $n$ polynomials $P_1',\ldots,P_n'$ given by the columns of $\mathtt{Q}\cdot\KB$ have $Q_0$ as a common factor. In particular, this implies that they have a common root $[u_0:v_0]$. Multiplying the matrix $A_{P_1',\ldots,P_n'}$ on the left by $(u^{2n-1}_0,u^{2n-2}_0v_0,\cdots,u_0v^{2n-2}_0,v^{2n-1}_0)$, we get 
    \[(u_0P_1'(u_0,v_0), v_0P_1'(u_0,v_0),\cdots,u_0P_n'(u_0,v_0),v_0P_n'(u_0,v_0)).\]
    Since $P_i'(u_0,v_0)=0$ for $i=1,\ldots,n$, we have a nontrivial element in the kernel of the transpose of $A_{P_1',\ldots,P_n'}$, and therefore $\det{A_{P_1',\ldots,P_n'}}=\mc{A}(B,Q)=0$.
\end{proof}

The next proposition states that our two regular maps have the same vanishing locus (assuming that $\det{\VB}\neq 0$).

\begin{prop}\label{prop: same zero locus}
    The polynomials $\mc{A}(B,Q)$ and $ (\det{\VB})^{2n}\cdot \mc{R}(B,Q)$ have the same zero locus in $\mb{A}^{2\binom{n}{2}}_{\overline{k(\ell)}}\times\mb{A}^9_{\overline{k}(\ell)}$.
\end{prop}
\begin{proof}
    Note that $\mc{R}(B,Q) = 0$ if and only if $ B\cap Q\neq \emptyset$. Indeed, $\mc{R}(B,Q)=0$ if and only if there exists a point $b\in B(\overline{k})$ for which $Q_1-b_xQ_0$ and $Q_2-b_yQ_0$ have a common root, which occurs if and only if there is a point $[u_0:v_0]\in\mathbb P^1(\overline{k})$ for which $Q_1(u_0,v_0) - b_xQ_0(u_0,v_0) = 0$ and $Q_2(u_0,v_0) - b_yQ_0(u_0,v_0) = 0$. Dividing by $Q_0(u_0,v_0)$, we find that this condition is equivalent to the existence of $[u_0:v_0]\in \P^1_{\overline{k}}$ such that $Q(u_0,v_0) = [1:b_x:b_y]$, which shows that $p\in Q(\overline{k})$. 
    
    By Lemma~\ref{lem: det(VB) =0 implies ABQ=0}, if $\det{\VB} = 0$, then $\detABQ=0$. It thus remains to show that if $\det{\VB}\neq 0$, then $\detABQ = 0$ if and only if $Q\cap B\neq \emptyset$. To this end, assume that we have $b = [1:b_x:b_y]\in (B\cap Q)(\overline{k})$. Then the coordinates of the composition map $g_B\circ Q=\gamma$ have a common root $[u_0:v_0]$, and thus the proof of Lemma~\ref{lem: det(VB) =0 implies ABQ=0} gives us $\detABQ = 0$.

    Conversely, if $\detABQ=0$, then there exist linear polynomials $L_1,\cdots,L_n$ (in $u,v$) such that $\sum_{i=1}^n L_iP_i'=0$ by \cite[Lemma 3.2.2.(3)]{FinashinKharlamov2021}. It follows that $\sum_{i=1}^n L_i\cdot g^{B,\beta'}\circ[Q_0:Q_1:Q_2]=0$, where $g^i_{B,\beta'}$ denotes the $i\textsuperscript{th}$ coordinate of the embedding $g_{B,\beta'}$, which corresponds to the basis element $\beta'_i$. In particular, the conic $Q$ is contained in the curve $Y:=\sum_{i=1}^n L_i\beta_i'$, which has degree $n$ and contains the locus $B$. If $B\cap Q=\varnothing$, then all points of $B$ are contained in the components of $Y$ away from $Q$, so $B$ is contained in a curve of degree $n-2$. But $\det{\VB}\neq 0$, and hence there is no such curve. We thus have $B\cap Q\neq\varnothing$, as desired.
    \end{proof}

Before giving examples of $B$ and $Q$ for which $\mc{A}(B,Q)=(\det\VB)^{2n}\cdot\mc{R}(B,Q)\neq 0$, we need the following lemma about elementary symmetric polynomials.

\begin{lemma}\label{lem:symmetric}
    Let $e_i(X_1,\ldots,X_n)$ denote the $i\textsuperscript{th}$ elementary symmetric polynomial in $n$ variables. Adopt the convention that $e_0=1$, and $e_i(X_1,\ldots,X_n)=0$ if $i<0$ or $i>n$. Then for any $1\leq j\leq n$, we have
    \[e_{i+1}(X_1,\ldots,X_n)=\sum_{z=0}^{i+1} e_z(X_1,\ldots,X_j)e_{i-z+1}(X_{j+1},\ldots,X_n).\]
\end{lemma}
\begin{proof}
    Let
    \[E_z:=e_z(X_1,\ldots,X_j)e_{i-z+1}(X_{j+1},\ldots,X_n).\]
    Each of the summands $E_z$ consists of a sum of distinct monomials of degree $i+1$, each with coefficient 1. Moreover, given a monomial summand $M:=X_1^{c_1}\cdots X_n^{c_n}$ of $e_{i+1}(X_1,\ldots,X_n)$ (so $c_\ell\in\{0,1\}$ and $\sum_{\ell=1}^n c_\ell=i+1$), there is at most one $z$ such that $M$ is a summand of $E_z$. Indeed, any monomial summand $X_1^{d_1}\cdots X_n^{d_n}$ of $E_z$ must satisfy $\sum_{\ell=1}^j d_\ell=z$. 
    
    It therefore suffices to show that the numbers of monomial summands in $\sum_{z=0}^{i+1} E_z$ and $e_{i+1}(X_1,\ldots,X_n)$ are equal. The latter number is $\binom{n}{i+1}$, essentially by definition of the elementary symmetric polynomials. Similarly, the number of summands in $e_z(X_1,\ldots,X_j)$ and $e_{i-z+1}(X_{j+1},\ldots,X_n)$ are given by $\binom{j}{z}$ and $\binom{n-j}{i-z+1}$, respectively. Since the indeterminates for these two elementary symmetric polynomials are disjoint, the product $E_z$ consists of $\binom{j}{z}\binom{n-j}{i-z+1}$ distinct monomials. Summing up to count the distinct monomials of $\sum_{z=0}^{i+1} E_z$, we find that it suffices to prove
    \[\binom{n}{i+1}=\sum_{z=0}^{i+1}\binom{j}{z}\binom{n-j}{i-z+1}.\]
    This follows from the binomial theorem by comparing the coefficient of $x^{i+1}$ on both sides of the equation
    \[(x+1)^n=(x+1)^j(x+1)^{n-j}.\]
    Equivalently (but perhaps more illustratively), repeated application of the standard identity $\binom{n-m}{i-z}+\binom{n-m}{i-z+1}=\binom{n-m+1}{i-z+1}$ implies that the following rows all have the same sum:
    \[\begin{tikzcd}[sep=tiny]
        &&&& \binom{n}{i+1} &&&&\\
        &&& \binom{n-1}{i+1} &+& \binom{n-1}{i} &&&\\
        && \binom{n-2}{i+1} &+& 2\binom{n-2}{i} &+& \binom{n-2}{i-1} &&\\
        &\binom{n-3}{i+1} &+& 3\binom{n-3}{i} &+& 3\binom{n-3}{i-1} &+& \binom{n-3}{i-2}&\\
        \binom{n-4}{i+1} &+& 4\binom{n-4}{i} &+& 6\binom{n-4}{i-1} &+& 4\binom{n-4}{i-2} &+& \binom{n-4}{i-3}.\qedhere
    \end{tikzcd}\]
\end{proof}

\begin{prop}
\label{prop: agree at somethint nonzero}
    Let $a_1,\dots,a_n\in\overline{k(\ell)}$ be distinct elements. Define 
    \begin{equation}\label{eq: choice B}
        B:=\left\{[1:a_i:a_j]\in \mathbb P^2\ \middle |\ 1\leq i<j\leq n\right\},
    \end{equation}
    and $Q_1=Q_2=u^2$ and $Q_0=v^2$. For these choices of $B$ and $Q$, we have $\mc{A}(B,Q) = \det(\VB)^{2n}\cdot\mc{R}(B,Q)\neq 0$.
\end{prop}
\begin{proof}
    We will prove this in four steps.
   \begin{enumerate}[Step 1:]
   \item $\mc{R}(B,Q)=\prod_{i<j}(a_i-a_j)^2$. We prove this step directly. For each point $[1:a_i:a_j]\in B$, we have $Q_1-a_iQ_0=u^2-a_iv^2$ and $Q_2-a_jQ_0=u^2-a_jv^2$. Thus
   \begin{align*}
       \mc{R}(B,Q)&=\prod_{1\leq i<j\leq n}\Res(u^2-a_iv^2,u^2-a_jv^2)\\
       &=\prod_{1\leq i<j\leq n}(a_i-a_j)^2.
   \end{align*}
   \item $\det{\VB}\neq 0$. We prove this step by contradiction. If $\det{\VB}=0$, then there exists a non-zero polynomial $f(x,y)$ of degree $n-2$ such that $f(a_i,a_j)=0$ for all $1\leq i<j\leq n$. If we fix $x=a_1$, we get a polynomial $f(a_1,y)$ of degree at most $n-2$ with $n-1$ distinct roots ($y=a_2,\ldots,a_n$). It follows that $f(a_1,y)=0$, so $x-a_1$ is a factor of $f(x,y)$. Let $f_1(x,y)=\frac{f(x,y)}{x-a_1}$. Then $f_1(a_1,y)$ is a polynomial of degree at most $n-3$ with $n-2$ distinct roots, and hence $x-a_2$ is a factor of $f_1(x,y)$. Repeating this process, we find that $x-a_i$ is a factor of $f(x,y)$ for all $i$, and hence $f(x,y)=0$ (as the $\deg(f)<n-1$). This contradicts the fact that $f$ was a non-zero polynomial, so we find that $\det{\VB}\neq 0$.
   \item With our particular choice of $[Q_0:Q_1:Q_2]$ and $B$, we have 
   \[A_{P_1',\ldots,P_n'}=\det\VB\cdot\begin{pmatrix}
       1 & 0 & \cdots & 1 & 0\\
       0 & 1 & \cdots & 0 & 1\\
       p_{2n-4,1} & 0 & \cdots & p_{2n-4,n} & 0\\
       0 & p_{2n-4,1} & \cdots & 0 & p_{2n-4,n}\\
       \vdots & \vdots & \ddots & \vdots & \vdots\\
       p_{0,1} & 0 & \cdots & p_{0,n} & 0\\
       0 & p_{0,1} & \cdots & 0 & p_{0,n}
   \end{pmatrix},\]
   where
   \begin{align*}
       p_{j,i}=\begin{cases}(-1)^{\ell-1} e_\ell(a_1,\ldots,\hat{a}_{n-i+1},\ldots,a_n) & j=2n-2-2\ell,\\
       0 & j\text{ odd}\end{cases}
   \end{align*}
   and $e_\ell(X_1,\ldots,X_{n-1})$ is the $\ell\textsuperscript{th}$ elementary symmetric polynomial in $n-1$ variables. (Here, we use the conventions that $e_0=1$ and $\hat{a}_i$ means that $a_i$ is omitted.) 
   
   Proving this formula is the lengthiest of our four steps. To begin, we need to construct our basis $\beta'$ from the columns of
   \[\KB=\begin{pmatrix}-\VB^*\cdot\RB\\ \det\VB\cdot\mathtt{I}_n\end{pmatrix}.\]
   We can then compute $(p_{j,i})=\mathtt{Q}\cdot\KB$, where $\mathtt{Q}:k[X,Y,Z]_{(n-1)}\to k[u,v]_{(2n-2)}$ is the linear transformation given by
   \[\mathtt{Q}(X^iY^jZ^{n-1-i-j})=u^{2(i+j)}v^{2(n-1-i-j)}.\]
   Since $\det\VB\neq 0$, we have $\VB^*=\det\VB\cdot\VB^{-1}$, and hence
   \[\KB=\det\VB\cdot\begin{pmatrix}-\VB^{-1}\cdot\RB\\ \mathtt{I}_n\end{pmatrix}.\]
   We may therefore compute the columns of $\KB$ by computing the columns of
   \[\KB':=\begin{pmatrix}-\VB^{-1}\cdot\RB\\ \mathtt{I}_n\end{pmatrix}.\]
   From now on, we interpret the columns of $\KB'$ as polynomials in $x,y$ (rather than homogeneous polynomials in $x,y,z$) by setting $z=1$. As identity matrices are easy enough to understand, we turn our attention to $\VB^{-1}\cdot\RB$. For each column $w_j$ of $\RB$ (see Equation~\ref{equation: matrix RB}), we need to solve the system $\VB\cdot f_j=w_j$ for the column $f_j$. That is, we need to find a degree $n-2$ polynomial $f_j\in k(\ell)[x,y]$ (which we conflate with its column of coefficients) such that $f_j(b_{i,x},b_{i,y})=w_{i,j}$ for all $[1:b_{i,x}:b_{i,y}]\in B$, where $w_{i,j}$ denotes the $i\textsuperscript{th}$ row of $w_j$. It follows that in order to solve $f_j(b_{i,x},b_{i,y})=w_{i,j}$, it suffices to find a polynomial $f_j$ such that
   \[g_j(x,y):=f_j(x,y)-x^{n-j}y^{j-1}\]
   satisfies $g_j(b_{i,x},b_{i,y})=0$ for all $1\leq i\leq\binom{n}{2}$. Let $r^j_{s,t}\in k(\ell)$ be such that
   \[f_j(x,y)=\sum_{s=0}^{n-2}\left(\sum_{t=0}^s r^j_{s,t}y^t\right)x^{n-2-s}.\]
   Let $r^j_s(y)=\sum_{t=0}^s r^j_{s,t}y^t$, so that $f_j(x,y)=\sum_{s=0}^{n-2}r^j_s(y)x^{n-2-s}$.

   We need to compute $r^j_{s,t}$ such that $g_j(b_{i,x},b_{i,y})=0$ for all $i$. We will compute $r^1_{s,t}$ and $r^2_{s,t}$ to illustrate our general approach, after which we will compute $r^j_{s,t}$ for all $j$. Recall that for our choice of $B$, we have
   \begin{align*}
       (b_{1,x},b_{1,y})&=(a_1,a_2),\\
       (b_{n-1,x},b_{n-1,y})&=(a_1,a_n),\\
       (b_{n,x},b_{n,y})&=(a_2,a_3),\\
       (b_{2n-3,x},b_{2n-3,y})&=(a_2,a_n),\\
       &\ \ \vdots\\
       (b_{\binom{n}{2},x},b_{\binom{n}{2},y})&=(a_{n-1},a_n).
   \end{align*}

   For $j=1$, we have $g_1(x,y)=-x^{n-1}+f_1(x,y)$. If $g_1(b_{i,x},b_{i,y})=0$ for all $i$, then $g_1(x,a_n)$ is a degree $n-1$ polynomial in $x$ with roots $a_1,\ldots,a_{n-1}$. By Vieta's formulas, it follows that $r^1_{0,0}=a_1+\ldots+a_{n-1}$. Note that for any fixed $y_0$, the roots of $g_1(x,y_0)$ must also sum to $a_1+\ldots+a_{n-1}$. Thus the roots of $g_1(x,a_{n-1})$ are given by $a_1,\ldots,a_{n-2}$ (by our definition of $B$) and also $a_{n-1}$ (since the roots must sum to $a_1+\ldots+a_{n-1}$). Vieta's formulas now give us
   \begin{align*}
   -\sum_{1\leq i<j\leq n-1}a_ia_j&=r^1_{1,0}+r^1_{1,1}a_n\\
   &=r^1_{1,0}+r^1_{1,1}a_{n-1},
   \end{align*}
   so $r^1_{1,1}=0$ and $r^1_{1,0}=-\sum_{i<j}a_ia_j$. Continuing this process for $y=a_{n-2},a_{n-3},\ldots$, we conclude that
   \[r^1_s(y)=(-1)^s e_{s+1}(a_1,\ldots,a_{n-1}),\]
   where $e_{s+1}$ denotes the $(s+1)\textsuperscript{st}$ elementary symmetric polynomial in $n-1$ variables.

   For $j=2$, we have $g_2(x,y)=(r^2_{0,0}-y)x^{n-2}+\sum_{s=1}^{n-2}r^2_s(y)x^{n-2-s}$. As with $j=1$, we will consider $g_2(x,a_n)$. We then have $n-1$ roots ($a_1,\ldots,a_{n-1}$) of this degree $n-2$ polynomial, so $g_2(x,a_n)$ must be identically zero. It follows that $r^2_s(a_n)=0$ for all $s$, so each of these polynomials is divisible by $y-a_n$. Let $\bar{r}^2_s(y)=\frac{r^2_s(y)}{y-a_n}$. Factoring, we have
   \[g_2(x,y)=(y-a_n)\left(-x^{n-2}+\sum_{s=1}^{n-2}\bar{r}^2_s(y)x^{n-2-s}\right).\]
   Now since $a_i\neq a_n$ for $i<n$, the assumption that $g_2(a_i,a_n)=0$ implies that the polynomial
   \[-x^{n-2}+\sum_{s=1}^{n-2}\bar{r}^2_s(a_{n-1})x^{n-2-s}\]
   has roots $a_1,\ldots,a_{n-2}$. Repeating our arguments from the $j=1$ case, we find that $\bar{r}^2_s(y)=(-1)^s e_s(a_1,\ldots,a_{n-2})$, where we now have the $s\textsuperscript{th}$ elementary symmetric polynomial in $n-2$ variables (rather than the $(s+1)\textsuperscript{st}$ elementary symmetric polynomial in $n-1$ variables). In particular, we have $r^2_{s}=(-1)^s(y-a_n)e_s(a_1,\ldots,a_{n-2})$ and hence
   \begin{align*}
       r^2_{s,t}=\begin{cases}
           (-1)^{s+1}a_{n}e_s(a_1,\ldots,a_{n-2}), & t=0,\\
           (-1)^{s}e_s(a_1,\ldots,a_{n-2}) & t=1,\\
           0 & t\geq 2.
       \end{cases}
   \end{align*}

   For $j>2$, we follow the same steps as in the $j=2$ case. We find that $a_1,\ldots,a_{n-1}$ are roots of $g_j(x,a_n)$, so $g_j(x,a_n)=0$. Thus $r^j_{0,0}=0$, so $g_j(x,a_n)$ has degree $n-3$ in $x$. This implies that $g_j(x,a_{n-1})=0$, which means that the linear polynomial $r_1^j(y)$ has distinct roots $a_{n}$ and $a_{n-1}$ and must therefore be identically zero, and likewise for the polynomials $r^j_s(y)$ with $s<j-2$ upon repeating this process. We conclude that
   \[g_j(x,y)=\prod_{i=0}^{j-2}(y-a_{n-i})\cdot\left(-x^{n-j}+\sum_{s=j-1}^{n-2}\bar{r}^j_s(y)x^{n-2-s}\right),\]
   where
   \begin{equation}\label{eq:rbar}
   \bar{r}^j_s(y)\cdot\prod_{i=0}^{j-2}(y-a_{n-i})=r^j_s(y).
   \end{equation}
   Since $g_j(x,a_{n-j+1})=0$ and $\prod_{i=0}^{j-2}(a_{n-j+1}-a_{n-i})\neq 0$, it follows that the degree $n-j$ polynomial
   \[-x^{n-j}+\sum_{s=j-1}^{n-2}\bar{r}^j_s(a_{n-j+1})x^{n-2-s}\]
   has roots $a_1,\ldots,a_{n-j}$. As in the $j=2$ case, we find that
   \[\bar{r}^j_s(y)=(-1)^s e_{s-j+2}(a_1,\ldots,a_{n-j}).\]
   It now follows from Equation~\ref{eq:rbar} and Vieta's formulas that
   \begin{equation}\label{eq:r_st}
    r^j_{s,t}=(-1)^{s+t+1}e_{s-j+2}(a_1,\ldots,a_{n-j})e_{j-1-t}(a_{n-j+2},\ldots,a_n),
   \end{equation}
   where we use the convention that $e_0=1$ and $e_m=0$ for $m<0$.

   We have solved for the columns of $\VB^{-1}\cdot\RB$. When considering the columns of $\KB'$, we must multiply the columns of $\VB^{-1}\cdot\RB$ by $-1$. After composing with $[Q_0:Q_1:Q_2]$ (i.e.~multiplying by $\mathtt{Q}$), the monomials corresponding to the coefficients of the lower block $\mathtt{I}_n$ are all $u^{2n-2}$ for our choice of conic. It follows that
   \[\mathtt{Q}\cdot\KB'=\begin{pmatrix}
       1 & 1 & \cdots & 1\\
       0 & 0 & \cdots & 0\\
       -\sum_{s=0}^{n-2}r^1_{s,s} & -\sum_{s=0}^{n-2}r^2_{s,s} & \cdots & -\sum_{s=0}^{n-2}r^n_{s,s}\\
       0 & 0 & \cdots & 0\\
       -\sum_{s=0}^{n-3}r^1_{s+1,s} & -\sum_{s=0}^{n-3}r^2_{s+1,s} & \cdots & -\sum_{s=0}^{n-3}r^n_{s+1,s}\\
       \vdots & \vdots & \ddots & \vdots\\
       0 & 0 & \cdots & 0\\
       -(r^1_{n-3,0}+r^1_{n-2,1}) & -(r^2_{n-3,0}+r^2_{n-2,1}) & \cdots & -(r^n_{n-3,0}+r^n_{n-2,1})\\
       0 & 0 & \cdots & 0\\
       -r^1_{n-2,0} & -r^2_{n-2,0} & \cdots & -r^n_{n-2,0}
   \end{pmatrix}.\]
   
   To complete Step 3, it remains to show that 
   \[-\sum_{s=0}^{n-2-i}r^j_{s+i,s}=(-1)^ie_{i+1}(a_1,\ldots,\hat{a}_j,\ldots,a_n).\]
   By Equation~\ref{eq:r_st}, we have
   \[r^j_{s+i,s}=(-1)^{i+1}e_{s+i-j+2}(a_1,\ldots,a_{n-j})e_{j-1-s}(a_{n-j+2},\ldots,a_n).\]
   Note that $r^j_{s+i,s}=0$ when $s<j-i-2$ or $s>j-1$, so
   \[-\sum_{s=0}^{n-2-i}r^j_{s+i,s}=-\sum_{s=j-i-2}^{j-1}r^j_{s+i,s}.\]
   We will re-index by setting $z=s-(j-i-2)$, so that
   \begin{align*}
       -\sum_{s=j-i-2}^{j-1}r^j_{s+i,s}&=-\sum_{z=0}^{i+1}r^j_{z+j-2,z+j-i-2}\\
       &=(-1)^{i+2}\sum_{z=0}^{i+1}e_z(a_1,\ldots,a_{n-j})e_{i-z+1}(a_{n-j+2},\ldots,a_n).
   \end{align*}
   It follows from Lemma~\ref{lem:symmetric} that
   \[\sum_{z=0}^{i+1}e_z(a_1,\ldots,a_{n-j})e_{i-z+1}(a_{n-j+2},\ldots,a_n)=e_{i+1}(a_1,\ldots,\hat{a}_{n-j+1},\ldots,a_n),\]
   as required.
   \item $\detABQ=(\det{\VB})^{2n}\cdot\prod_{i<j}(a_i-a_j)^2$. By Step 3, we have
   \begin{align*}
       \mc{A}(B,Q)&=\det{A_{P_1',\ldots,P_n'}}\\
       &=(\det\VB)^{2n}\cdot\det\begin{pmatrix}
       1 & 0 & \cdots & 1 & 0\\
       0 & 1 & \cdots & 0 & 1\\
       p_{2n-4,1} & 0 & \cdots & p_{2n-4,n} & 0\\
       0 & p_{2n-4,1} & \cdots & 0 & p_{2n-4,n}\\
       \vdots & \vdots & \ddots & \vdots & \vdots\\
       p_{0,1} & 0 & \cdots & p_{0,n} & 0\\
       0 & p_{0,1} & \cdots & 0 & p_{0,n}
   \end{pmatrix},
   \end{align*}
   so it suffices to prove that the latter matrix (which we call $\tilde{N}$) has determinant $\prod_{i<j}(a_i-a_j)^2$ . Note that $\tilde{N}=N\otimes\mathtt{I}_2$, where
   \[N:=\begin{pmatrix}
       1 & 1 & \cdots & 1\\
       p_{2n-4,1} & p_{2n-4,2} & \cdots & p_{2n-4,n}\\
       \vdots & \vdots & \ddots & \vdots\\
       p_{0,1} & p_{0,2} & \cdots & p_{0,n}
   \end{pmatrix},\]
   so $\det\tilde{N}=\det{N}^2$. It therefore suffices to prove that
   \[\det{N}=\pm\prod_{1\leq i<j\leq n}(a_i-a_j),\]
   which we will do by induction on $n$. Note that $p_{i,j}$ is a polynomial in $n-1$ variables, even though this dependence on $n$ is not reflected in the notation. To make this dependence more explicit, let $e_\ell(\bm{a}_{i,n})=e_\ell(a_1,\ldots,\hat{a}_i,\ldots,a_n)$ and
   \[N_n=\begin{pmatrix}
       1 & 1 & \cdots & 1\\
       e_1(\bm{a}_{n,n}) & e_1(\bm{a}_{n-1,n}) & \cdots & e_1(\bm{a}_{1,n})\\
       \vdots & \vdots & \ddots & \vdots\\
       (-1)^{n-2}e_{n-1}(\bm{a}_{n,n}) & (-1)^{n-2}e_{n-1}(\bm{a}_{n-1,n}) & \cdots & (-1)^{n-2}e_{n-1}(\bm{a}_{1,n})
   \end{pmatrix}.\]
   If $n=2$, then we have
   \[N_2=\begin{pmatrix}1 & 1\\ a_2 & a_1\end{pmatrix},\]
   and thus $\det{N_2}=a_1-a_2$, as desired. Next, assume that 
   \[\det{N_n}=\pm\prod_{1\leq i<j\leq n}(a_i-a_j)^2.\]
   In order to compute $\det{N_{n+1}}$, subtract the first column of $N_{n+1}$ from the last $n$ columns of $N_{n+1}$. Since
   \begin{align*}
       e_\ell(\bm{a}_{i,n+1})-e_\ell(\bm{a}_{n+1,n+1})&=(a_{n+1}-a_i)e_{\ell-1}(\bm{a}_{i,n}),
   \end{align*}
   it follows that $\det{N_{n+1}}$ is equal to
   \[\det\begin{psmallmatrix}
           1 & 0 & \cdots & 0\\
           e_1(\bm{a}_{n+1,n+1}) & (a_{n+1}-a_n)e_0(\bm{a}_{n,n}) & \cdots & (a_{n+1}-a_1)e_0(\bm{a}_{1,n})\\
           \vdots & \vdots & \ddots & \vdots\\
           (-1)^{n-1}e_n(\bm{a}_{n+1,n+1}) & (-1)^{n-1}(a_{n+1}-a_n)e_{n-1}(\bm{a}_{n,n}) & \cdots & (-1)^{n-1}(a_{n+1}-a_1)e_{n-1}(\bm{a}_{1,n})
    \end{psmallmatrix}.\]
    After factoring $a_{n+1}-a_i$ out of the $(n+2-i)\textsuperscript{th}$ column for $i=1,\ldots,n$ and a $-1$ out of last $n-1$ rows, the resulting matrix has the form
    \[\begin{pNiceArray}{c|ccc}
    1 & 0 & \cdots & 0 \\
    \hline
    * & \Block{3-3}<\large>{N_n} && \\
    \vdots &&&\\
    * &&&
    \end{pNiceArray}.\]
    It follows that
    \begin{align*}
    \det{N_{n+1}}&=(-1)^{n-1}\prod_{i=1}^n(a_{n+1}-a_i)\cdot\det{N_n}\\
    &=-\prod_{i=1}^n(a_i-a_{n+1})\cdot\det{N_n}.
   \end{align*}
   We thus have
   \[\det{N_{n+1}}=\mp\prod_{1\leq i<j\leq n+1}(a_i-a_j),\]
   which completes the proof by induction.\qedhere
   \end{enumerate}
   \end{proof}

Our final goal is to show that $\mc{A}(B,Q)$ and $(\det\VB)^{2n}\cdot\mc{R}(B,Q)$ differ by a constant, which we will then show to be one. Before proving this, we need the following lemmas.

\begin{lemma} \label{lem: resultant irred}
    As a polynomial on $\A^{2\binom n2}\times \A^9$ (in the coordinates of $B$ and coefficients of $Q$), the polynomials
    \[\Res(Q_1-b_xQ_0,Q_2-b_yQ_0)\]
    are irreducible over $\overline{k(\ell)}$. Here, $(b_x,b_y)$ represents one of the $(b_{i,x},b_{i,y})$.
\end{lemma}
\begin{proof}
    We can directly compute these resultants using the following Sage code.
    \begin{lstlisting}
R.<b_x,b_y,q_00,q_01,q_02,\\
   q_10,q_11,q_12,q_20,q_21,q_22> = QQ[];
S.<u> = R[];

Q_0 = q_00 + q_01*u + q_02*u^2;
Q_1 = q_10 + q_11*u + q_12*u^2;
Q_2 = q_20 + q_21*u + q_22*u^2;

f = Q_1 - b_x*Q_0;
g = Q_2 - b_y*Q_0;

print(f.resultant(g))
\end{lstlisting}
    Let $T=\overline{k(\ell)}[q_{00},\ldots,q_{22},b_x]$ with fraction field $K$. It is then straightforward to check that $h(b_y):=\Res(Q_1-b_xQ_0,Q_2-b_yQ_0)\in T[b_y]$ is a quadratic polynomial in $b_y$. Moreover, if we write $h(b_y)=h_2b_y^2+h_1b_y+h_0$, we can compute that $\gcd(h_0,h_1,h_2)=1$ in the ring $T$. This can be done by noting that the gcd of the individual terms of $h_2$ is 1, and that no terms of $h_2$ are divisible by $b_x$, but some terms of $h_0$ and $h_1$ are divisible by $b_x$. Alternatively, one can compute this gcd with Sage by adding the following to the code given above:
\begin{lstlisting}
res = f.resultant(g)

h_2 = res.coefficient(b_y^2)
h_1 = res.coefficient(b_y)
h_0 = res - h_2*b_y^2 - h_1*b_y

print(gcd(h_0,h_1))
\end{lstlisting}    
    Using the quadratic formula (and assuming $\op{char}{k}\neq 2$), it is straightforward to check that $h(b_y)$ is irreducible over $K[b_y]$. It now follows from Gauß's lemma that $h(b_y)$ is irreducible over $T[b_y]$, as desired.
\end{proof}

\begin{lemma}\label{lem:det irred}
    The polynomial $\det\VB$ in the coordinates $b_{i,x}$ and $b_{i,y}$ is irreducible over $\overline{k(\ell)}$. 
\end{lemma}
\begin{proof}
    Recall that we have assumed $n\geq 3$ (Remark~\ref{rem:n geq 3}), so $\binom{n}{2}\geq 3$. In the notation of \cite[Theorem 1.5]{DT09}, we have $N=\binom{n}{2}$,
    \begin{align*}
        (\gamma_{1_1},\gamma_{2_1},\ldots,\gamma_{N_1})&=(0,1,0,2,1,0,\ldots,n-2,n-3,\ldots,0),\\
        (\gamma_{1_2},\gamma_{2_2},\ldots,\gamma_{N_2})&=(0,0,1,0,1,2,\ldots,0,1,\ldots,n-2),
    \end{align*}
    and $\overline{\gamma}=(0,0)$. Since $\gamma_2=(1,0)$, the largest natural number $d_\Gamma$ such that $\frac{1}{d_\Gamma}(\gamma_2-\overline{\gamma})=(\frac{1}{d_\Gamma},0)$ is an element of $\mb{N}^2$ is 1. Finally, since $\gamma_2=(1,0)$ and $\gamma_3=(0,1)$, it follows that $\dim\mc{L}_\Gamma=2$. Therefore the assumptions of \cite[Theorem 1.5]{DT09} hold, which implies that the interpolation determinant $\det\VB$ is irreducible over any algebraically closed field.
\end{proof}

We can now show that $\mc{A}(B,Q)$ and $\mc{R}(B,Q)$ differ by $c\cdot(\det\VB)^d$ for some constant $c$ and some integer $d$.

\begin{prop}\label{prop:equal up to scalar}
    As polynomials on $\A^{2\binom n2}\times \A^9$, we have $\mc{A}(B,Q) = c\cdot\det(\VB)^d\cdot\mc{R}(B,Q)$ for some $c\in\overline{k(\ell)}$ and $d\in\mb{Z}$.
\end{prop}
\begin{proof}
Since $\mc{A}$ and $\det(\VB)^{2n}\cdot\mc{R}$ have the same zero locus (Proposition~\ref{prop: same zero locus}), Hilbert's Nullstellensatz implies that radical ideals generated by the two polynomials are the same. Therefore
\begin{equation}\label{eq: Nullstellensatz A and R}
\sqrt{(\mc{A})} =\sqrt{\left(\det\VB\right)}\cap \sqrt{(\mc{R})}.
\end{equation}

Since $\mc{R}$ is the product of distinct irreducible polynomials (over $\overline{k(\ell)}$) and $\det\VB$ is irreducible (over $\overline{k(\ell)}$) by Lemmas \ref{lem: resultant irred} and \ref{lem:det irred}, we conclude that the ideals $(\mc{R})$ and $(\det\VB)$ are radical. This implies that
\begin{equation}\label{eq: prime dec of A}
    \mc{A} = c\cdot\det(\VB)^d\cdot\prod_{b\in B(\overline{k(\ell)})} \Res(Q_1 - b_x Q_0, Q_2 - b_y Q_0)^{N_i}
\end{equation}
for some integers $d>0$ and $N_i>0$ and some $c\in\overline{k(\ell)}$. 

We finish by showing that $N_i = 1$ for $i=1,\ldots, \binom{n}{2}$. Consider the value of the polynomial $\mc{A}$ for $Q$ and $\lambda Q$ for some $\lambda\in\overline{k(\ell)}$. By Proposition~\ref{prop:independence for t} (iii), we have  $\Res(\lambda Q_1 - \lambda b_x Q_0, \lambda Q_2 - \lambda b_y Q_0) = \lambda^4  \Res(Q_1 - b_x Q_0, Q_2 - b_y Q_0)$. Similarly, by Proposition~\ref{prop:independence for a} (iii), we have $\mc{A}(\lambda Q,B) = \lambda^{2n(n-1)}\mc{A}(Q,B)$. Comparing the multiplicity of $\lambda$ in Equation~\ref{eq: prime dec of A}, we get
\[4\sum_{i=1}^{\binom n2}N_i = 2n(n-1),\]
which implies that $\sum_i N_i=\binom{n}{2}$. Since $N_i>0$ for each $i$, we have $N_i=1$ for all $i$.
\end{proof}

Finally, we conclude by using Proposition~\ref{prop: agree at somethint nonzero} to show that $c=1$ and $d=2n$ in Equation~\ref{eq: prime dec of A}.

\begin{cor}\label{cor:equal}
    We have that $c=1$ and $d=2n$ in Equation \ref{eq: prime dec of A}. Therefore, $\mathcal A(B,Q) =(\det\VB)^{2n}\cdot \mathcal R(B,Q)$, which implies that the conic index is equal to the local index, as desired.
\end{cor}
\begin{proof}
    Proposition \ref{prop: agree at somethint nonzero} and Equation \ref{eq: prime dec of A} imply that 
    \[c\cdot(\det\VB)^d = (\det\VB)^{2n}\]
    for any $B$ of the form given in Equation~\eqref{eq: choice B}. It follows that $(\det\VB)^{2n-d}$ is constant for any such $B$, which is true only if $2n-d=0$. Consequently, $c=1$.
\end{proof}

\bibliography{segre}{}
\bibliographystyle{alpha}
\end{document}